\documentclass[11pt,a4paper,leqno]{amsart}

\usepackage{common}

\title[On entangled commutators]{On entangled and multi-parameter commutators}
\author{Kangwei Li}
\author{Henri Martikainen}

\address[K.L.]{Center for Applied Mathematics, Tianjin University, Weijin Road 92, 300072 Tianjin, China}
\email{kli@tju.edu.cn}

\address[H.M.]{Department of Mathematics and Statistics, Washington University
in St. Louis, 1 Brookings Drive, St. Louis, MO 63130, USA}
\email{henri@wustl.edu}

\makeatletter
\@namedef{subjclassname@2010}{%
    \textup{2010} Mathematics Subject Classification}
\makeatother

\subjclass[2010]{42B20}
\keywords{Calder\'on--Zygmund operators, singular integrals, multi-parameter
    analysis, commutators, compactness, Zygmund dilations}

\begin{document}

\allowdisplaybreaks

\begin{abstract}
	We complement the recent theory of general singular integrals $T$ invariant under
	the Zygmund dilations $(x_1, x_2, x_3) \mapsto (s x_1, tx_2, st x_3)$
	by proving necessary and sufficient conditions for the boundedness and compactness
	of commutators $[b,T]$ from $L^p \to L^q$. Previously, only the $p=q$ upper bound in terms
	of a Zygmund type little $\BMO$ space was known for general operators,
	and it appears that there has been some confusion about the corresponding lower bound in recent literature.
	We give complete characterizations whenever $p \le q$
	for a general class of non-degenerate Zygmund type singular integrals. Some
	of the results are somewhat surprising in view of existing papers -- for instance,
	compactness always forces $b$ to be constant. Even in the simpler situation
	of bi-parameter singular integrals it appears that this has not been observed previously.
\end{abstract}

\maketitle

\section{Introduction}
This paper is continuation for the recent papers Hyt\"onen--Li--Martikainen--Vuorinen \cite{HLMV}
and Airta--Li--Martikainen \cite{ALM24} concerning singular integral operators (SIOs)
invariant under the simplest known entangled dilations, the Zygmund dilations
$(x_1, x_2, x_3) \mapsto (s x_1, t x_2, st x_3)$.
Here invariance under dilations means that the kernel $K$
of the underlying SIO $Tf(x) = \int K(x,y)f(y)\ud y$ satisfies suitable invariant estimates --
for instance, for the standard size estimate $|K(x,y)| \lesssim \frac{1}{|x-y|^d}$
in $\R^d$ the natural invariance is with respect to the
one-parameter dilations $x \mapsto tx$.
On the other hand,
in the classical multi-parameter setup $\R^d = \R^{d_1} \times \cdots \times \R^{d_m}$, $d_1 + \cdots + d_m = d$,
the dilations are independent (non-entangled) in each parameter:
$x = (x_1, \ldots, x_m) \mapsto (t_1x_1, \ldots, t_m x_m)$.
Modern developments from the past decade place such product space theory
roughly on the same level as the one-parameter one. However, there is still a lot that
we do not understand about entangled dilations. In this paper we answer some of the most natural
questions about commutators $[b,T]$ in the Zygmund dilation setting in $\R^3$,
simultaneously shedding new light to some particular pure multi-parameter questions.

SIOs with entangled dilations, on a general level, have their origin in the work of Nagel-Wainger \cite{NW}.
They proved the $L^2$ boundedness of a class of operators with kernels invariant under both translations and dilations
of general entangled type. In particular,
the three dimensional convolution kernel
$$
	K(x):=\frac{\sign(x_1x_2)}{(x_1x_2)^2+x_3^2}=(st)^2K(\rho_{s,t}(x)),
$$
where $\rho_{s,t}(x) = (sx_1, tx_2, stx_3)$, is the so-called Nagel--Wainger kernel --
the most basic example of an SIO kernel that is invariant under Zygmund dilations
$\rho_{s,t}$ in the above strong sense. Zygmund dilations
also appear naturally as they are compatible with the group law of the Heisenberg
group, see M\"uller--Ricci--Stein \cite{MRS}. After Nagel--Wainger, Ricci--Stein \cite{RS}
dealt with the $L^p$ boundedness relaxing also some of the invariance assumptions -- however,
the kernels there are somewhat indirectly described using a suitable series representation.
This has then been further generalized by Fefferman--Pipher \cite{FEPI} to cover some optimal
weighted estimates for Zygmund type multipliers, which are smooth enough.
Finally, Han--Li--Lin--Tan \cite{HLLT} considered general convolution kernels, but only in the unweighted case.
This was, nevertheless, an important step in recognizing how the abstract theory of Zygmund SIOs should look like
structurally (how kernel assumptions should be formulated).

The general theme of research for us, and others, in this setting has been to
obtain various optimal $L^p$ estimates allowed by the Zygmund invariance --
in particular, one wants to beat the standard tri-parameter estimates
by exploiting carefully the additional Zygmund invariance. Thus, just plain $L^p$
estimates for $T$ itself are not the most interesting ones as they do not allow us to improve
over the tri-parameter estimates, and allow the analysis to venture outside of the Zygmund context,
for instance, by using the tri-parameter maximal function instead of the more refined Zygmund
maximal function.
However, weighted estimates $L^p(w) \to L^p(w)$ are very interesting -- the general hope is that
the invariances should allow such estimates for a larger class of weights
than the tri-parameter $A_p$ weights. Proving them should also require
the careful exploitation of the Zygmund invariance. Fefferman--Pipher \cite{FEPI}
introduced the Zygmund weights
\begin{equation*}
	[w]_{A_{p,Z}} :=\sup_{R\in\mathcal{R}_Z}
	\Big(\frac{1}{|R|}\int_R w(x)\ud x\Big)\Big(\frac{1}{|R|} \int_R w^{-1/(p-1)}(x) \ud x\Big)^{p-1}
	<\infty, \qquad 1 < p < \infty,
\end{equation*}
where the supremum is over the {\em Zygmund rectangles} $R = I_1\times I_2\times I_3$, $\ell(I_3) = \ell(I_1)\ell(I_2)$,
for this very purpose. For the more restricted tri-parameter
weights the supremum would run over all rectangles -- for examples
of the Zygmund weights see \cite{HLMV}.
Such optimal weighted estimates are already non-trivial for maximal functions and, in fact,
particularly deep if one would try to go much beyond the Zygmund dilations. Nevertheless, Fefferman--Pipher
proved $L^p(w) \to L^p(w)$ estimates, whenever $w \in A_{p, Z}$, for the natural Zygmund
maximal function and also for the so-called Fefferman--Pipher multipliers. Every Ricci--Stein operator \cite{RS}
is a sum of two Fefferman--Pipher multipliers.

However, the story turned out to be much more complicated, but to see this we must consider
general SIOs. Indeed, there is a precise threshold for the kernel estimates:
if the kernel decay in terms of the deviation of $z\in\R^3$ from the ``Zygmund manifold'' $\abs{z_1z_2}=\abs{z_3}$
is not fast enough, general SIOs invariant under Zygmund dilations can fail to be bounded with Zygmund weights.
It turns out that the Fefferman--Pipher multipliers just barely have enough decay for this.
We constructed counterexamples and showed the delicate positive result in the optimal range in \cite{HLMV},
and also pushed beyond the convolution setting of \cite{HLLT}.
The decay from the Zygmund manifold is measured using the quantity
$$
	D_{\theta}(t_1, t_2, t_3) := \Big(\frac{t_3}{t_1t_2} + \frac{t_1t_2}{t_3} \Big)^{-\theta}, \, 0 < \theta \le 1,
$$
and it turned out that $\theta < 1$ is insufficient for weighted bounds.

Now, what does this have to do with commutators $[b,T]f := bTf - T(bf)$? Quite a lot actually. First,
the commutator question is philosophically similar to the weighted one: can we beat the tri-parameter
results and allow more general symbols $b$ if $T$ is Zygmund and not just tri-parameter?
Second, at least when it comes to sufficient conditions for the boundedness $[b,T] \colon L^p \to L^p$,
one can use Zygmund weighted estimates of $T$ in connection with the classical
Cauchy integral trick. This was actually done in \cite{DLOPW-ZYGMUND} -- however, this limits
the amount of SIOs the commutator estimates are valid for, since, as discussed, it was
realized in \cite{HLMV} that such weighted estimates often fail. However, we managed
to prove in \cite{ALM24} that
$$
	\|b\|_{\bmo_Z} := \sup_{R \in \mathcal{R}_Z} \frac{1}{|R|}\int_{R}|b(x) -\ave{b}_R| \ud x < \infty,
$$
is sufficient without any such kernel decay restrictions that apply to the weighted case ($\theta < 1$ is fine for this).
In the current paper we prove the necessity of this condition, with general $\theta$, for a natural class of non-degenerate Zygmund SIOs,
and consider similar questions in the off-diagonal $L^p \to L^q$, as well as analogous compactness questions.
The weighted theory was surprising in \cite{HLMV} -- there appears to be some surprises here as well.

First, the necessity of $b \in \bmo_Z$ was also discussed in \cite{DLOPW-ZYGMUND}.
In \cite{DLOPW-ZYGMUND} it is first showed that $b \in \bmo_Z$ is necessary for the
Nagel--Wainger kernel -- our result recovers this by showing that $b \in \bmo_Z$
is actually necessary for a wide class of non-degenerate Zygmund SIOs (see Theorem \ref{thm:ppbddnec}),
and the Nagel--Wainger kernel is a particular instance of such kernels.
But a contrasting result is stated as Theorem 1.5 in \cite{DLOPW-ZYGMUND} saying
that $b \in \bmo_Z$ is \textbf{not} a necessary condition for \textbf{any}
Fefferman--Pipher multiplier or Ricci--Stein \cite{RS}
operator.
However, our result seems to point to the opposite being true.
Indeed, we know from \cite{HLMV} that all of these operator classes fall into
the general Zygmund SIO framework of \cite{HLMV}, and we now prove the necessity
of $b \in \bmo_Z$ for all such general SIOs that are, in addition, suitably non-degenerate.
We discuss this situation further in the Appendix \ref{app:FP}.

Our next result concerns the sufficient and necessary conditions for boundedness in the off-diagonal
$L^p \to L^q$, $p < q$. While the diagonal question is the most well-known in the classical situation,
the off-diagonal questions also have a long history and interesting applications -- we refer, for instance,
to the introduction of \cite{Hyt-Com} for a good account of the boundedness theory and their applications.
When $T$ is a multi-parameter SIO, Oikari showed in \cite{Oik2022bipar}
that $b$ must be constant in this case. The result is more interesting in the Zygmund situation --
$b$ does have to be constant in $x_1, x_2$ but is allowed to be H\"older continuous
with exponent $2\alpha$, $\alpha := 1/p-1/q$, and thus non-constant if $\alpha \le 1/2$
(see Theorem \ref{thm:equidefoff} and Theorem \ref{thm:suffbddoff} for the precise statements).
We also comment that we prove our upper bounds (for $p < q$ and $p=q$) so that they only involve
the so-called off-diagonal constants -- a weaker notion of $\|[b,T]\|_{L^p \to L^q}$.
Such a style is familiar from \cite{Hyt-Com} and \cite{AHLMO2021}. We do not attempt to tackle
$p > q$ as the necessity result is only known in one-parameter \cite{Hyt-Com} but not
in bi-parameter \cites{AHLMO2021, Oik2022bipar} situations, and there currently are well-known
sparse related obstructions to the multi-parameter proofs (see \cite{AHLMO2021}).

The compactness theory of commutators also has a long history -- already Uchiyama \cite{Uch1978}
characterized the compactness $[b, T] \colon L^p \to L^p$ in terms of a vanishing mean oscillation
space $\VMO \subset \BMO$ -- this was, of course, in the classical one-parameter situation.
Multiple extensions have since been done, see, for instance, the introduction of \cite{HOS2023} for a good account.
In this paper we prove that compactness is non-interesting for $[b,T]$ in the
Zygmund situation -- however, the proof of this is interesting.
That is, we prove that $b$ must be a constant for $[b,T]$ to be compact $L^p \to L^q$ whenever $p \le q$
(see Theorem \ref{thm:comdiag} and Theorem \ref{thm:compactnessoff}).
Interestingly, an easier version of our argument also gives that $b$ must be constant in the
bi-parameter case -- that is, when $T$ is a bi-parameter SIO.
Already the bi-parameter version can be seen as somewhat surprising -- compare to
Lacey--Li \cite{Lacey-Li-JMAA}, where a $\operatorname{vmo}$ type characterization is
stated in the bi-parameter Riesz transform case. But their
vanishing mean oscillation type space requires only that oscillations vanish when $\diam(R) \to 0$ --
we obtain that oscillations needs to actually vanish even when only one side length goes to $0$, and this
forces the symbol to be constant. We also note that
our result seems to be consistent with some results concerning big Hankel operators
on the polydisk \cites{AYZ, CotSad}.

The above also means that it appears to us that for at least these multi-parameter commutators
$[b,T]$ it is not meaningful
to study finer properties of compactness, such as, when $[b,T]$
is in the Schatten class (results in this direction have appeared in the first version of the preprint \cite{Lacey-Li-Wick-2023} for the bi-parameter
Hilbert transform, but the authors have informed us that they have now removed them).
In contrast, compactness yields non-trivial $\VMO$ type spaces for the harder bi-commutator $[T_1, [b, T_2]]$ -- see
Martikainen--Oikari \cite{Martikainen-Oikari-2024} for recent work in this direction. That is, there are
multi-parameter commutator settings
where compactness is meaningful -- see also Lacey--Terwilleger--Wick \cite{LaTeWi06}
for earlier work related to the compactness of these bi-commutators. So in these situations two one-parameter SIOs are commuted
to yield bi-parameter setups,
but there is no obvious analogue of this in entangled situations.

We discuss the proofs. The necessity results rely on the so-called approximate weak factorization (AWF)
method introduced in the one-parameter setting in \cite{Hyt-Com} and then later used in pure multi-parameter situations e.g.
in \cites{Martikainen-Oikari-2024, Oik2022bipar, AHLMO2021}. This is the first time it is used
in an entangled situation, but the Zygmund AWF needed here is not particularly hard. The most interesting part
regarding the Zygmund AWF is our use of a new, weaker than the tri-parameter,
notion of non-degeneracy that allows us to run the AWF
only in Zygmund rectangles. Once the relatively routine AWF is done, the necessity result in the diagonal takes a line or two.
This is because $\bmo_Z$ is a space defined using oscillations -- that is the language that the AWF method understands.
However, we have to separately characterize the H\"older-continuous functions in the third variable
with Zygmund oscillations to handle the off-diagonal $p < q$. The sufficiency in the case $p=q$ was already proved
in \cite{ALM24}, but we prove the sufficiency in the $p<q$ case in this paper -- it is a bit tricky partly because
it appears that a pointwise
representation of the commutator $[b,T]$ is non-trivial to prove in the Zygmund setting
(see, for instance, Section 6 in \cite{AHLMO2021} to understand what goes into proving
such pointwise identities in the more classical situations; such tools are missing here).

The compactness results also rely on AWF, mostly via the refined scheme of the recent paper
\cite{Martikainen-Oikari-2024} (see also \cite{HOS2023}). This proof strategy yields that Zygmund oscillations have to vanish
in multiple situations -- one could then easily be tempted to define some kind of $\VMO$ space using these conditions,
they are so natural.
But it turns out that the conditions are such that they force $b$ to be a constant. This is quite easy
to deduce from the conditions
in the bi-parameter situation but somewhat trickier in the Zygmund case.

\subsection*{Notation} Throughout this paper,
$\langle f\rangle_E:= \frac 1{|E|}\int_E f$ for any measurable set $E$ of finite measure.
We also frequently use the notation $A\lesssim B$, which means that
$A\le CB$ for some irrelevant  constant $C>0$.
If $A\lesssim B$ and $B\lesssim A$ hold simultaneously, we simply denote it by $A\sim B$.

\subsection*{Acknowledgements}
K. Li was supported by the National Natural Science Foundation of China through
project numbers 12222114 and 12001400. This material is based upon work supported
by the National Science Foundation (NSF) under Grant No. 2247234 (H. Martikainen).
H.M. was, in addition, supported by the Simons Foundation through MP-TSM-00002361
(travel support for mathematicians).

H.M. thanks Tuomas Oikari and Tuomas Hyt\"onen for useful conversations.

\section{Approximate weak factorization methods and necessity results}\label{sec:awf}
\subsection{Non-degenerate kernels}
We now introduce what we mean with non-degeneracy in the Zygmund dilation setting.
First, for these considerations and approximate weak factorization (AWF)
we are going to need only the following upper bounds for our
kernel $K \colon (\R^3 \times \R^3) \setminus \Delta \to \C$, where
$$
	\Delta := \{(x,y) \in \R^3 \times \R^3 \colon x_i = y_i \textup{ for some } i\}.
$$
Let $0 < \theta \le 1$ -- we assume the size estimate
$$
	|K(x,y)| \lesssim \size_Z(|x_1-y_1|, |x_2-y_2|, |x_3-y_3|),
$$
where
\begin{align*}
	\size_Z(t_1, t_2, t_3)    & := \frac{D_{\theta}(t_1, t_2, t_3)}{t_1t_2t_3};                  \\
	D_{\theta}(t_1, t_2, t_3) & := \Big(\frac{t_3}{t_1t_2} + \frac{t_1t_2}{t_3} \Big)^{-\theta}.
\end{align*}
In addition, we will need the following weak variants of Zygmund type continuity assumptions.
Let $\omega \colon [0,1] \to [0, \infty)$ be increasing with $\omega(t) \to 0$ as $t \to 0$.
Let $|x_i - x_i'| \le |x_i-y_i| / 2$ for $i = 1,2,3$. Then we demand that
\begin{equation*}
	|K((x_1', x_2, x_3), y) - K(x,y)| \lesssim \omega\Big(\frac{|x_1-x_1'|}{|x_1-y_1|}\Big)
	\size_Z(|x_1-y_1|, |x_2-y_2|, |x_3-y_3|)
\end{equation*}
and
\begin{equation*}
	|K((x_1, x_2', x_3'), y) - K(x,y)| \lesssim
	\omega\Big(\frac{|x_2-x_2'|}{|x_2-y_2|} + \frac{|x_3-x_3'|}{|x_3-y_3|}\Big)
	\size_Z(|x_1-y_1|, |x_2-y_2|, |x_3-y_3|)
\end{equation*}
together with the symmetric conditions in the $y$-variable.
\begin{defn}[Non-degenerate Zygmund kernels]
	Let $K$ be a Zygmund invariant kernel on $\R^{3}$ as above.
	We say that $K$ is non-degenerate if
	for every $y \in\R^{3}$ and $\delta_1, \delta_2>0$ there exists
	$x \in \R^3$ with $|x_1-y_1| > \delta_1$, $|x_2-y_2| > \delta_2$
	and $|x_3-y_3| > \delta_1\delta_2$
	so that there holds
	\[
		|K(x, y)|\gtrsim \frac{1}{\delta_1^2\delta_2^2}.
	\]
\end{defn}
This is motivated by the following. Notice that
\begin{align*}
	\size_Z(t_1, t_2, t_3) & = \frac{1}{t_1t_2t_3} \Big(\frac{(t_1t_2)^2 + t_3^2}{t_1t_2t_3} \Big)^{-\theta} \\
	                       & = \frac{1}{(t_1^2t_2^2 + t_3^2)^{\theta}} \frac{1}{(t_1t_2t_3)^{1-\theta}}.
\end{align*}
So if $t_1 > \delta_1$, $t_2 > \delta_2$ and $t_3 > \delta_1\delta_2$, then
$$
	\size_Z(t_1, t_2, t_3) \le \frac{1}{(2\delta_1^2\delta_2^2)^{\theta}} \frac{1}{(\delta_1^2\delta_2^2)^{1-\theta}}
	\sim \frac{1}{\delta_1^2\delta_2^2}.
$$

Now, it follows that, in the definition of non-degeneracy, the point $x$ will automatically
satisfy $|x_1-y_1| \lesssim \delta_1$, $|x_2-y_2| \lesssim \delta_2$ and $|x_3-y_3| \lesssim \delta_1\delta_2$.
We show this next. First, to see $|x_1-y_1| \lesssim \delta_1$ we estimate
\begin{align*}
	\frac{1}{\delta_1^2\delta_2^2} & \lesssim |K(x,y)|                                                           \\
	                               & \lesssim \frac{1}{(|x_1-y_1|^2 \delta_2^2 + \delta_1^2\delta_2^2)^{\theta}}
	\frac{1}{(|x_1-y_1|\delta_1\delta_2^2)^{1-\theta}}                                                           \\
	                               & = \frac{1}{\delta_2^{2\theta}(|x_1-y_1|^2 + \delta_1^2)^{\theta}}
	\frac{1}{\delta_1^{1-\theta}\delta_2^{2}\delta_2^{-2\theta}|x_1-y_1|^{1-\theta}}                             \\
	                               & \sim \frac{1}{|x_1-y_1|^{2\theta}}
	\frac{1}{\delta_1^{1-\theta}\delta_2^{2}|x_1-y_1|^{1-\theta}}
	= \frac{1}{\delta_1^{1-\theta}\delta_2^{2}|x_1-y_1|^{1+\theta}}
\end{align*}
and notice that this implies $|x_1-y_1|^{1+\theta} \lesssim \delta_1^{1+\theta}$.
The proof of $|x_2-y_2| \lesssim \delta_2$ is symmetric. To see that $|x_3-y_3| \lesssim \delta_1\delta_2$
we notice that
\begin{align*}
	\frac{1}{\delta_1^2\delta_2^2} & \lesssim |K(x,y)|                                                \\
	                               & \lesssim \frac{1}{(\delta_1^2\delta_2^2 + |x_3-y_3|^2)^{\theta}}
	\frac{1}{\delta_1^{1-\theta}\delta_2^{1-\theta} |x_3-y_3|^{1-\theta}}                             \\
	                               & \sim \frac{1}{|x_3-y_3|^{2\theta}}
	\frac{1}{\delta_1^{1-\theta}\delta_2^{1-\theta} |x_3-y_3|^{1-\theta}}
	= \frac{1}{\delta_1^{1-\theta}\delta_2^{1-\theta} |x_3-y_3|^{1+\theta}}                           \\
\end{align*}
from which it follows that $|x_3-y_3|^{1+\theta} \lesssim \delta_1^{1+\theta}\delta_2^{1+\theta}$.

\subsection{Reflected Zygmund rectangles}
Let $I = I^1 \times I^2 \times I^3$ be a Zygmund rectangle in $\R^3$ --
that is, $\ell(I^3) = \ell(I^1)\ell(I^2)$. Let $c_{I^i}$ be the center of the interval $I^i$
and set $c_R = (c_{I^1}, c_{I^2}, c_{I^3})$. We use non-degeneracy with $y := c_R$,
$\delta_1 := A^{1/4}\ell(I^1)$ and $\delta_2 := A^{1/4}\ell(I^2)$, where
$A > 0$ is a large but at this point unspecified parameter. We find $x$
so that
$$
	|K(x, c_R)| \gtrsim \frac{1}{A} \frac{1}{\ell(I^1)^2\ell(I^2)^2} = \frac{1}{A}\frac{1}{\ell(I^1)\ell(I^2)\ell(I^3)}
	= \frac{1}{A|R|}.
$$
Let $\wt I^i$ be an interval centered at $c_{\wt I^i} := x_i$
and of side length $\ell(I^i)$. Then $c_{\wt R} = (c_{\wt I^1}, c_{\wt I^2}, c_{\wt I^3})$
is the center of the Zygmund rectangle $\wt R := \wt I^1 \times \wt I^2 \times \wt I^3$.
This is the reflected version of $R$. Notice that
\begin{align*}
	|c_{I^1} - c_{\wt I^1}| & \sim A^{1/4}\ell(I^1) \sim \dist(I^1, \wt I^1), \\
	|c_{I^2} - c_{\wt I^2}| & \sim A^{1/4}\ell(I^2) \sim \dist(I^2, \wt I^2), \\
	|c_{I^3} - c_{\wt I^3}| & \sim A^{1/2}\ell(I^3) \sim \dist(I^3, \wt I^3)
\end{align*}
and
$$
	|K(c_{\wt R}, c_R)| \sim \frac{1}{A|R|}.
$$

\begin{lem}
	Let $R$ be a Zygmund rectangle, $y \in R$ and $x \in \wt R$. Then we have
	(for some constant $C$) that
	$$
		|K(x,y) - K(c_{\wt R}, c_R)| \lesssim \omega(CA^{-1/4}) \frac{1}{A|R|}.
	$$
	In particular, for large $A$ we have
	$$
		\Big|\int_R K(x,y)\ud y \Big| \sim \int_R |K(x,y)| \ud y \sim \frac{1}{A}, \,\, x \in \wt R,
	$$
	and
	$$
		\Big|\int_{\wt R} K(x,y)\ud x \Big| \sim \int_{\wt R} |K(x,y)| \ud x \sim \frac{1}{A}, \,\, y \in R.
	$$
\end{lem}
\begin{proof}
	Notice that
	\begin{align*}
		|K(x,y) - K(c_{\wt R}, c_R)| & \le |K(x,y) - K((c_{\wt I^1}, x_{23}), y)|             \\
		                             & + |K((c_{\wt I^1}, x_{23}), y)-K(c_{\wt R}, y)|        \\
		                             & + |K(c_{\wt R}, y)-K(c_{\wt R}, (c_{I^1}, y_{23}))|    \\
		                             & + |K(c_{\wt R}, (c_{I^1}, y_{23}))-K(c_{\wt R}, c_R)|,
	\end{align*}
	where, e.g., $x_{23} := (x_2, x_3)$ for $x = (x_1, x_2, x_3)$.
	We deal with the first term -- the other being similar. Notice that
	$|x_1 - c_{\wt I^1}| \lesssim \ell(\wt I^1) = \ell(I^1)$
	and $|x_1-y_1| \ge \dist(I^1, \wt I^1) \gtrsim A^{1/4}\ell(I^1)$
	implies $|x_1-y_1| \ge 2|x_1 - c_{\wt I^1}|$ for big $A$, and so by the continuity
	estimates
	\begin{align*}
		|K(x,y) - K((c_{\wt I^1}, x_{23}), y)| & \lesssim \omega\Big(C\frac{\ell(I^1)}{A^{1/4}\ell(I^1)}\Big)
		\size_Z(|x_1-y_1|, |x_2-y_2|, |x_3-y_3|)                                                              \\
		                                       & \lesssim \omega(CA^{-1/4}) \frac{1}{A|R|}.
	\end{align*}
	The other three terms are similar proving the desired estimate for $|K(x,y) - K(c_{\wt R}, c_R)|$.
	To see the other claims, fix $x \in \wt R$ and notice that
	$$
		\int_R K(x,y) \ud y = |R| K(c_{\wt R}, c_R) + \int_R [K(x,y) - K(c_{\wt R}, c_R)]\ud y
	$$
	implies
	$$
		\int_R |K(x,y)| \ud y \ge \Big| \int_R K(x,y) \ud y\Big| \ge \frac{C_0}{A} - C_1\omega(CA^{-1/4}) \frac{1}{A}
		\gtrsim \frac{1}{A}
	$$
	if we choose $A$ so large that $\omega(CA^{-1/4})$ is small enough for the above estimate.
	On the other hand, it is trivial that
	$$
		\int_R |K(x,y)| \ud y \lesssim \int_R \frac{\ud y}{A|R|} = \frac{1}{A}.
	$$
	The other case is symmetric, and so we are done.
\end{proof}

\subsubsection{AWF for Zygmund rectangles}
In what follows (for the AWF considerations and also for the necessity results
for commutator boundedness) we do not really need a Zygmund singular integral.
Rather, we only need a non-degenerate kernel $K$ and then $T$ and $T^*$ are simply notation for the integrals
\begin{align*}
	Th(x) = \int_{\R^3} K(x, y) h(y)\ud y, \\
	T^* h(x) = \int_{\R^3} K(y, x)  h(y)\ud y.
\end{align*}
We only use these integrals in situations where there is plenty of separation,
such as, $x \in \wt R$ but $h$ is supported on $R$ (and so $y \in R$). That is,
we use these in off-diagonal situations.

\begin{lem}[One iteration AWF]
	Let $K$ be a non-degenerate Zygmund invariant kernel.
	Let $R = I^1 \times I^2 \times I^3$ be a Zygmund rectangle and $f \in L^1$ be a function
	that is supported on $R$ with $\int_R f = 0$. Then
	$$
		f = h_R T^* 1_{\wt R} - 1_{\wt R} Th_R + e_{\wt R},
	$$
	where
	\begin{enumerate}
		\item $|h_R| \lesssim A|f|$,
		\item $|e_{\wt R}| \lesssim \omega(CA^{-1/4})\langle |f| \rangle_R 1_{\wt R}$,
		\item $\int e_{\wt R} = 0$.
	\end{enumerate}
\end{lem}
\begin{proof}
	Define
	$$
		h_R := \frac{f}{T^* 1_{\wt R}}
	$$
	and simply write
	$$
		f = h_R T^* 1_{\wt R} = h_R T^* 1_{\wt R} - 1_{\wt R} T h_R + 1_{\wt R} T h_R
	$$
	so that $e_{\wt R} := 1_{\wt R} T h_R$. Notice that for $y \in \supp f \subset R$ we have
	$$
		|T^* 1_{\wt R}(y)| = \Big| \int_{\wt R} K(x, y) \ud x\Big| \sim \frac{1}{A}
	$$
	so that $h_R$ is not only well-defined but also $|h_R| \sim A|f|$.

	Next, for $x \in \wt R$ we have
	$$
		|e_{\wt R}(x)| = |T h_R(x)| = \Big|\int_R K(x,y)h_R(y) \ud y \Big|,
	$$
	where, for $y \in R$ we have
	$$
		h_R(y) = \frac{f(y)}{\int_{\wt R} K(x,y) \ud x} - \frac{f(y)}{\int_{\wt R} K(c_{\wt R},c_R) \ud x}
		+ \frac{f(y)}{|R|K(c_{\wt R},c_R)}.
	$$
	Notice first that
	\begin{align*}
		\Big|\int_R K(x,y)\frac{f(y)}{|R|K(c_{\wt R},c_R)}  \ud y \Big|
		 & \sim A \Big| \int_R [K(x,y) - K(c_{\wt R}, c_R)]f(y) \ud y\Big| \\
		 & \lesssim \omega(CA^{-1/4}) \langle |f|\rangle_R.
	\end{align*}
	On the other hand, we have
	\begin{align*}
		\Big|\frac{f(y)}{\int_{\wt R} K(x,y) \ud x} - \frac{f(y)}{\int_{\wt R} K(c_{\wt R},c_R) \ud x}\Big|
		 & = \frac{|f(y)| \Big|\int_{\wt R} [K(x,y) - K(c_{\wt R}, c_R)] \ud x \Big|}
		{\Big|\int_{\wt R} K(x,y) \ud x \Big| |R||K(c_{\wt R},c_R)|}                  \\
		 & \lesssim \frac{|f(y)| \omega(CA^{-1/4}) \frac{1}{A}}{\frac{1}{A^2}}
		= \omega(CA^{-1/4}) A |f(y)|
	\end{align*}
	so that
	\begin{align*}
		\Big|\int_R K(x,y)\Big[\frac{f(y)}{\int_{\wt R} K(x,y) \ud x} -
			\frac{f(y)}{\int_{\wt R} K(c_{\wt R},c_R) \ud x} \Big]  \ud y \Big|
		\lesssim \omega(CA^{-1/4}) \langle |f|\rangle_R.
	\end{align*}
	We have shown $|e_{\wt R}| \lesssim \omega(CA^{-1/4})\langle |f| \rangle_R 1_{\wt R}$.

	Finally, notice that
	$$
		\int e_{\wt R} = \int 1_{\wt R} T\Big( \frac{f}{T^*1_{\wt R}}\Big)
		= \int T^* 1_{\wt R} \frac{f}{T^*1_{\wt R}} = \int f = 0.
	$$
\end{proof}
We want a few variants and corollaries of this for later use.
It is useful to iterate this one more time to make the error term $e$
be supported in $R$ instead of $\wt R$.
\begin{lem}[twice iterated AWF]
	Let $K$ be a non-degenerate Zygmund invariant kernel.
	Let $R = I^1 \times I^2 \times I^3$ be a Zygmund rectangle and $f \in L^1$ be a function
	that is supported on $R$ with $\int_R f = 0$. Then
	\begin{align*}
		f & = [h_R T^* 1_{\wt R} - 1_{\wt R} Th_R] +
		[h_{\wt R} T 1_{R} - 1_{R} T^*h_{\wt R}]
		+ e_{R},
	\end{align*}
	where
	\begin{enumerate}
		\item $|h_R| \lesssim A|f|$,
		\item $|h_{\wt R}| \lesssim A\langle |f|\rangle_R 1_{\wt R}$,
		\item $|e_{R}| \lesssim \omega(CA^{-1/4})\langle |f| \rangle_R 1_{R}$,
		\item $\int e_{R} = 0$.
	\end{enumerate}
\end{lem}
\begin{proof}
	Use the one iteration AWF first to $f$ like above, and then to
	$e_{\wt R}$ with $1_{\wt R}$ replaced by $1_R$ and $T$ with $T^*$.
\end{proof}

Next, we will bound oscillations
$$
	\osc(b, R) := \frac{1}{|R|}\int_R |b-\langle b \rangle_R|,
$$
where $R$ is a Zygmund rectangle, using the AWF.
\begin{lem}\label{lem:AWFCorBdd}
	Let $K$ be a non-degenerate Zygmund invariant kernel.
	Let $b \in L^{1}_{\loc}$ and the constant $A$ be large enough. Then for all
	Zygmund rectangles $R$ there exists functions
	$$
		\varphi_{R, 1}, \, \varphi_{R,2}, \,
		\psi_{\wt R, 1}, \, \psi_{\wt R, 2}
	$$
	so that for $i = 1,2$ we have
	$$
		|\varphi_{R, i}|\lesssim 1_{R} \qquad \textup{and} \qquad |\psi_{\wt R,i}|\lesssim |R|^{-1}1_{\wt R}
	$$
	and
	$$
		\osc(b, R) \lesssim \sum_{i=1}^2 |\langle [b,T]\varphi_{R,i}, \psi_{\wt R, i}\rangle|.
	$$
\end{lem}
\begin{proof}
	Choose $f$ with $\|f\|_{L^{\infty}} \le 1$, $\supp f \subset R$ and $\int_R f = 0$ so that
	$$
		|R|\osc(b, R) =
		\int_R |b-\langle b \rangle_R| \lesssim |\langle b, f\rangle|.
	$$
	Write using the twice iterated AWF that
	$$
		\langle b, f\rangle = -\langle [b, T]h_R, 1_{\wt R}\rangle
		+ \langle [b,T]1_R, h_{\wt R}\rangle
		+ \langle b, e_R\rangle,
	$$
	where
	\begin{align*}
		|\langle b, e_R\rangle| & \lesssim \omega(CA^{-1/4}) \langle |f| \rangle_R \int_R |b-\langle b \rangle_R| \\
		                        & \le \omega(CA^{-1/4}) |R| \osc(b,R).
	\end{align*}
	This shows that
	\begin{align*}
		\osc(b,R) & \lesssim |\langle [b, T]h_R, |R|^{-1}1_{\wt R}\rangle| \\
		          & + |\langle [b,T]1_R, |R|^{-1}h_{\wt R}\rangle|
		+ \omega(CA^{-1/4}) \osc(b,R).
	\end{align*}
	Fixing $A$ large enough we get
	$$
		\osc(b,R) \lesssim |\langle [b, T]\varphi_{R, 1}, \psi_{\wt R, 1}\rangle|
		+ |\langle [b,T]\varphi_{R, 2}, \psi_{\wt R, 2}\rangle|,
	$$
	where
	\begin{align*}
		\varphi_{R, 1} := h_R \qquad                 & \textup{and} \qquad \varphi_{R, 2} := 1_R, \\
		\psi_{\wt R, 1} := |R|^{-1} 1_{\wt R} \qquad & \textup{and} \qquad \psi_{\wt R, 2}
		:= |R|^{-1} h_{\wt R}.
	\end{align*}
\end{proof}
The constant $A$ is now \textbf{fixed} large enough so that conclusions like above hold.

\subsection{Necessity results for boundedness}
We define an off-diagonal constant that can be used instead of the full
norm of the commutator $[b,T]$. That is, for these boundedness considerations
we only need the kernel $K$ and the off-diagonal constants instead of
an actual SIO and the related commutator.
\begin{defn}
	For $u, t \in (1,\infty)$ and $B(x,y) = b(x)-b(y)$ define
	\begin{align*}
		\Off_{u}^{t} =  \sup \frac{1}{|P_1|^{1+1/u-1/t}}
		\Big| \iint_{\R^{3}\times \R^3} B(x,y)
		K(x,y) f_1(y)f_2(x) \ud y \ud x \Big|,
	\end{align*}
	where the supremum is taken over Zygmund rectangles $P_1 = J^1 \times J^2 \times J^3$
	and $P_2 = L^1 \times L^2 \times L^3$ with
	$$
		\ell(J^i) = \ell(L^i) \qquad \textup{and} \qquad
		\dist (J^i, L^i)\sim \ell(J^i)
	$$
	and over functions $f_i \in L^{\infty}(P_i)$ with $\|f_j\|_{L^{\infty}} \le 1$.
\end{defn}
\begin{lem}\label{lem:off_est}
	Let $K$ be a non-degenerate Zygmund invariant kernel and $b \in L^1_{\loc}$. Then we have
	$$
		\osc(b, R) \lesssim \Off_u^t |R|^{1/u-1/t}, \qquad 1 < u,t < \infty,
	$$
	for all Zygmund rectangles $R$.
\end{lem}
\begin{proof}
	Using AWF as above, estimate
	$$
		\osc(b, R) \lesssim \sum_{i=1}^2 |R|^{-1}|\langle [b,T]\varphi_{R,i}, |R|\psi_{\wt R, i}\rangle|,
	$$
	where $\|\varphi_{R, i}\|_{L^{\infty}} \lesssim 1$ and $\||R| \psi_{\wt R, i}\|_{L^{\infty}} \lesssim 1$,
	and the functions $\varphi_{R, i}, \psi_{\wt R, i}$ are supported on $R$ and $\wt R$, respectively.
	It follows directly from the definition that
	$$
		|R|^{-1}|\langle [b,T]\varphi_{R,i}, |R|\psi_{\wt R, i}\rangle| \lesssim \Off_u^t |R|^{1/u-1/t}
	$$
	so we are done.
\end{proof}

We now use this to quickly obtain a necessary condition for $\Off_p^p < \infty$,
in particular, for $\|[b, T]\|_{L^p \to L^p} < \infty$, $1 < p < \infty$.
Let $\calR_Z$ denote the collection of Zygmund rectangles.
Define
$$
	\|b\|_{\bmo_Z} := \sup_{R \in \calR_Z} \osc(b,R).
$$
\begin{thm}\label{thm:ppbddnec}
	Let $K$ be a non-degenerate Zygmund invariant kernel and $b \in L^1_{\loc}$. Then we have
	$$
		\|b\|_{\bmo_Z} \lesssim \Off_p^p, \qquad 1 < p < \infty.
	$$
\end{thm}
\begin{proof}
	By the previous lemma
	$$
		\osc(b, R) \lesssim \Off_p^p
	$$
	for all Zygmund rectangles, and so we are done.
\end{proof}
The corresponding commutator upper bound for a Zygmund CZO $T$ was recently proved in \cite{ALM24}.
As weighted estimates with Zygmund weights fail for $T$ if $\theta < 1$, see \cite{HLMV}, this had
to be done via the representation theorem \cite{HLMV} and not using the usual Cauchy integral trick.
In any case, for non-denegerate Zygmund SIOs we have $\|[b, T]\|_{L^p \to L^p} \sim \|b\|_{\bmo_Z}$.
As we will later have to consider actual Zygmund SIOs/CZOs (and not just these off-diagonal constants
with very weak kernels like in this section otherwise), we will also later recall more carefully what
the Zygmund SIOs/CZOs that appear in these upper bounds are. In any case, boundedness in the diagonal
is completely characterized.

In the off-diagonal situation $L^p \to L^q$, $q > p$, we require the following space.
Define
\begin{align*}
	\|b\|_{\bmo_Z^{\alpha}} & := \sup_{R\in \mathcal R_Z} \calO_{\alpha}(b,R), \\
	\calO_{\alpha}(b,R)     & := |R|^{-\alpha} \osc(b, R).
\end{align*}
\begin{thm}\label{thm:equidefoff}
	Let $K$ be a non-degenerate Zygmund invariant kernel and $b \in L^1_{\loc}$. Then we have
	$$
		\|b\|_{\bmo_Z^{{\alpha}}} \lesssim \Off_p^q, \qquad 1 < p < q < \infty,
	$$
	where
	$$
		\alpha := \frac{1}{p} - \frac{1}{q} > 0.
	$$
	Moreover, we have the following characterization of $\bmo_Z^{\alpha}$ for any
	$\alpha > 0$. If $b \in L^s_{\loc}$ for $s > 1$ then there holds that
	$$
		\|b\|_{\bmo_Z^{\alpha}}
		\sim \sup_{x, y \colon x_3 \ne y_3} \frac{|b(x)-b(y)|}{|x_3-y_3|^{2\alpha}}.
	$$
\end{thm}
\begin{proof}
	We only need to prove the characterization of $\bmo_Z^{\alpha}$ as the first claim
	directly follows from Lemma \ref{lem:off_est}.

	Denote
	$$
		C :=  \sup_{x, y \colon x_3 \ne y_3} \frac{|b(x)-b(y)|}{|x_3-y_3|^{2\alpha}}.
	$$
	Notice that for a Zygmund rectangle $R = I^1 \times I^2 \times I^3$
	we have for $x \in R$ that
	$$
		|b(x)-\langle b \rangle_R| \le C\ell(I^3)^{2\alpha} = C|R|^{\alpha},
	$$
	and so $\osc(b, R) \le C|R|^{\alpha}$ showing that $\|b\|_{\bmo_Z^{\alpha}} \le C$.

	We first aim to show the, perhaps weaker appearing, claim
	\begin{equation}\label{eq:weakerHolder}
		|b(x)-b(y)| \lesssim \|b\|_{\bmo_Z^{\alpha}}\big(|x_1-y_1|^{2\alpha} |x_2-y_2|^{2\alpha}
		+ |x_3-y_3|^{2\alpha}\big).
	\end{equation}
	To do this we fix $x,y$. We begin by showing that
	if $|x_3-y_3| \le |x_1-y_1||x_2-y_2|$ then
	$$
		|b(x)-b(y)| \lesssim \|b\|_{\bmo_Z^{\alpha}} |x_1-y_1|^{2\alpha} |x_2-y_2|^{2\alpha}.
	$$
	To this end, we define, given $\eps > 0$, that
	$$
		r_1 := \max(|x_1-y_1|, \eps), \,\,\, r_2 := \max(|x_2-y_2|, \eps), \,\,\, r_3 := r_1r_2.
	$$
	(The $\eps > 0$ is just so that we do not have to deal with subcases like $x_1 = y_1$ here.)
	We will be defining a sequence of Zygmund rectangles $I_k$ and $J_k$, $k = 1,2,\ldots$,
	centered at $x$ and $y$, respectively, so that $I_{k+1} \subset I_k$, $J_{k+1} \subset J_k$
	and the side lengths converge to zero. In addition, we will define
	a top Zygmund rectangle $L$ so that $I_1 \subset L$ and $J_1 \subset L$,
	and we will set $I_0 = J_0 := L$.
	In detail, let $L$ be the Zygmund rectangle centered
	at $((x_1+y_1)/2,(x_2+y_2)/2, (x_3+y_3)/2)$ with side lengths
	$2r_1$, $2r_2$ and $4r_3$.
	We then define the Zygmund rectangle
	$I_k = I_k^1 \times I_k^2 \times I_k^3$, $k = 1,2, \ldots$, by setting
	\begin{align*}
		I_k^1 & := (x_1 - 2^{-k}r_1, x_1 + 2^{-k}r_1),             \\
		I_k^2 & := (x_2 - 2^{-k}r_2, x_2 + 2^{-k}r_2),             \\
		I_k^3 & := (x_3 - 2^{-2k+1}r_1r_2, x_3 + 2^{-2k+1}r_1r_2).
	\end{align*}
	Define $J_k$ analogously.
	Next, we record an easy general fact which shows that, indeed, $I_1, J_1 \subset L$.
	If $r > 0$ and $z, w \in \R^d$ are such that $|z-w| \le r$, then
	the balls
	$$
		B := B(z,r/2) \qquad \textup{and} \qquad \wt B := B\Big(\frac{z+w}{2}, r \Big)
	$$
	satisfy $B \subset \wt B$.

	Recall that in this part of the proof we are assuming $b \in L^s_{\loc}$ for $s > 1$.
	We have the Lebesgue's differentiation theorem over rectangles by
	the boundedness of the tri-parameter maximal function
	on $L^s$, and so
	$$
		b(x) = \lim_{k \to \infty} \langle b \rangle_{I_k}
		= \sum_{k=1}^\infty ( \langle b\rangle_{I_{k}}- \langle b\rangle_{I_{k-1}} )+ \langle b\rangle_{L}
	$$
	if $x$ is a Lebesgue point of $b$, in particular, almost everywhere.
	Assuming $x,y$ are such Lebesgue points we then have the above
	identity and
	$$
		b(y) = \lim_{k\to \infty} \langle b\rangle_{J_k}
		= \sum_{k=1}^\infty ( \langle b\rangle_{J_{k}}- \langle b\rangle_{J_{k-1}} )+ \langle b\rangle_{L}.
	$$
	Subtracting these identities gives
	\begin{align*}
		|b(x)-b(y)| & \le \sum_{k=1}^\infty | \langle b\rangle_{I_{k}}- \langle b\rangle_{I_{k-1}} |+ \sum_{k=1}^\infty | \langle b\rangle_{J_{k}}- \langle b\rangle_{J_{k-1}} | \\
		            & \lesssim  \sum_{k=1}^\infty \big(\osc(b, I_{k-1})+\osc(b, J_{k-1})\big)
		\lesssim  \|b\|_{\bmo_Z^{\alpha}} (r_1r_2)^{2\alpha},
	\end{align*}
	where we used that $|I_k| \sim 2^{-4k}(r_1r_2)^2$.
	Finally, letting $\eps \to 0$ we get
	\[
		|b(x)-b(y)|\lesssim \|b\|_{\bmo_Z^{\alpha}} |x_1-y_1|^{2\alpha}|x_2-y_2|^{2\alpha}
	\]
	whenever $|x_3-y_3| \le |x_1-y_1||x_2-y_2|$ and $x,y$ are, in addition, Lebesgue points.

	We then prove that, if $|x_1-y_1| |x_2-y_2| < |x_3-y_3|$ and $x,y$ are Lebesgue points, we have
	$$
		|b(x) - b(y)| \lesssim \|b\|_{\bmo_Z^{\alpha}} |x_3-y_3|^{2\alpha}.
	$$
	This is a quite similar argument as above. We first consider the case $x_1 = y_1$ and $x_2 = y_2$.
	Let $r := |x_3-y_3| > 0$. Define the Zygmund rectangle
	$I_k = I_k^1 \times I_k^2 \times I_k^3$, $k = 1,2, \ldots$, by setting
	\begin{align*}
		I_k^1 & := (y_1 - 2^{-k}r^{1/2}, y_1 + 2^{-k}r^{1/2}), \\
		I_k^2 & := (y_2 - 2^{-k}r^{1/2}, y_2 + 2^{-k}r^{1/2}), \\
		I_k^3 & := (x_3 - 2^{-2k+1}r, x_3 + 2^{-2k+1}r).
	\end{align*}
	Define $J_k$ analogously (the only difference being in $J_k^3$).
	Let $L$ be the Zygmund rectangle centered
	at $(y_1, y_2, (x_3+y_3)/2)$ with side lengths
	$(2r)^{1/2}$, $(2r)^{1/2}$ and $2r$, and set $I_0 = J_0 := L$. Repeating
	the argument from above with these choices then clearly gives
	$$
		|b(y_1, y_2, x_3) - b(y)| \lesssim \|b\|_{\bmo_Z^{\alpha}} r^{2\alpha}
	$$
	as desired. Suppose then $x_1 \ne y_1$ -- the case $x_2 \ne y_2$ is similar.
	The only difference is that now we use (in the notation of the first case)
	$r_1 = |x_1-y_1|$, $r_2 = |x_3-y_3|/|x_1-y_1|$ and $r_3 = r_1r_2 = |x_3-y_3|$.
	Again clear that the same strategy yields
	$$
		|b(x)-b(y)| \lesssim \|b\|_{\bmo_Z^{\alpha}} (r_1r_2)^{2\alpha}
		= \|b\|_{\bmo_Z^{\alpha}} |x_3-y_3|^{2\alpha}.
	$$

	We have now proved \eqref{eq:weakerHolder} for Lebesgue points $x,y$.
	We upgrade this now to all $x,y$ by modifying the function $b$ on a set of measure zero.
	For any $z$ that is not a Lebesgue point of $b$, by the density of the Lebesgue points
	we find a sequence of Lebesgue points $x_n\to z$. Then \eqref{eq:weakerHolder} implies that
	\[
		|b(x_{m})-b(x_n)|\lesssim |x_{m}^1-x_n^1|^{2\alpha}|x_{m}^2-x_{n}^2 |^{2\alpha}+ |x_{m}^3-x_n^3|^{2\alpha},
	\]
	and so $\{b(x_{n})\}_n$ is a Cauchy sequence. We then simply redefine
	\[
		b(z):= \lim_{n\to \infty} b(x_n).
	\]
	This is well-defined, since if $y_n\to z$ is another sequence of Lebesgue points, then
	\[
		|b(x_{n})-b(y_n)|\lesssim |x_{n}^1-y_n^1|^{2\alpha}|x_{n}^2-y_{n}^2 |^{2\alpha}+ |x_{n}^3-y_n^3|^{2\alpha},
	\]
	and so
	\[
		\lim_{n\to \infty} b(x_n)=\lim_{n\to \infty} b(y_n).
	\]
	With this redefinition it is easy to check that \eqref{eq:weakerHolder} holds
	for all $x,y$.

	Now, write
	\[
		b(x)= (b(x)-b(y_1, x_2, x_3))+ (b(y_1, x_2, x_3)- b(y_1, y_2, x_3))+ b(y_1, y_2, x_3).
	\]
	The first two terms are zero by \eqref{eq:weakerHolder}. Thus, we have
	$$
		b(x) = b(y_1, y_2, x_3),
	$$
	i.e., $b$ is a constant in the first two variables. With this the claim
	follows from \eqref{eq:weakerHolder}.
\end{proof}
Unlike in the diagonal case, the corresponding upper bound in this off-diagonal case
has not appeared previously. We will prove it in the next section.

\section{Sufficiency result in the off-diagonal}
We begin recalling more carefully what is a Zygmund SIO/CZO. Usually, we refer
the term Zygmund SIO for an operator $T$ that has a Zygmund kernel with full and partial
kernel estimates, but may not satisfy the cancellation assumptions of \cite{HLMV}.
A Zygmund CZO then means an SIO that satisfies the cancellation assumptions of \cite{HLMV},
in particular, one that
is $L^p \to L^p$, $1 < p < \infty$, bounded. Here we just state the kernel estimates
of an SIO for reference (as they are used below)
and refer the reader to \cite{HLMV} for a deeper understanding of the CZO part
and what it does and does not imply (with optimal weighted estimates being the deepest part).

An SIO $T$ is related to a full kernel $K$ in the following way.
The kernel $K$ is  a function
$
	K \colon (\R^3 \times \R^3) \setminus \Delta \to \C,
$
where
$$
	\Delta=\{(x,y) \in \R^3 \times \R^3 \colon |x_1-y_1||x_2-y_2||x_3-y_3|=0\}.
$$
If $f = f_1 \otimes f_{23}$ and $g = g_1 \otimes g_{23}$ with $f_1, g_1 \colon \R \to \C$
and $f_{23}, g_{23} \colon \R^2 \to \C$ having disjoint supports, then
$$
	\ave {Tf,g}
	= \iint K(x,y) f(y) g(x) \ud y \ud x.
$$
The kernel $K$ satisfies the following estimates.

Let $(x,y) \in (\R^3 \times \R^3)\setminus \Delta$. First, we assume that $K$ satisfies the size estimate
\begin{equation*}
	|K(x,y)|
	\lesssim \size_Z(|x_1-y_1|, |x_2-y_2|, |x_3-y_3|),
\end{equation*}
see the beginning of Section \ref{sec:awf} for the definition of $\size_Z$ and recall
that it depends on the parameter $\theta \in (0, 1]$.
Let then $x'=(x_1',x_2',x_3')$ be such that $|x_i'-x_i| \le |x_i-y_i|/2$ for $i=1,2,3$. We assume that $K$ satisfies the mixed size and H\"older estimates
\begin{equation*}
	\begin{split}
		|K((x'_1,x_2,x_3),y)-K(x,y)|
		 & \lesssim  \Big(\frac{|x_1'-x_1|}{|x_1-y_1|}\Big)^{\alpha_1}
		\size_Z(|x_1-y_1|, |x_2-y_2|, |x_3-y_3|),
	\end{split}
\end{equation*}
and
\begin{equation*}
	\begin{split}
		|K((x_1 & ,x_2',x_3'),y)-K(x,y)|
		\\
		        & \lesssim  \Big(\frac{|x_2'-x_2|}{|x_2-y_2|}+ \frac{|x_3'-x_3|}{|x_3-y_3|} \Big)^{\alpha_{23}}
		\size_Z(|x_1-y_1|, |x_2-y_2|, |x_3-y_3|),
	\end{split}
\end{equation*}
where $\alpha_1, \alpha_{23} \in (0,1]$.
Finally, we assume that  $K$ satisfies the H\"older estimate
\begin{equation*}
	\begin{split}
		|K(x', & y)-K((x_1',x_2,x_3),y)-K((x_1,x_2',x_3'),y)+K(x,y)|                                          \\
		       & \lesssim  \Big(\frac{|x_1'-x_1|}{|x_1-y_1|}\Big)^{\alpha_1}\Big(\frac{|x_2'-x_2|}{|x_2-y_2|}
		+\frac{|x_3'-x_3|}{|x_3-y_3|}\Big)^{\alpha_{23}}
		\size_Z(|x_1-y_1|, |x_2-y_2|, |x_3-y_3|).
	\end{split}
\end{equation*}

Define the adjoint kernels $K^*$, $K^*_1$ and $K^*_{2,3}$ via the following relations:
\begin{equation*}
	K^*(x,y)=K(y,x), \quad
	K^*_1(x,y)=K((y_1,x_2,x_3),(x_1,y_2,y_3))
\end{equation*}
and
\begin{equation*}
	K^*_{2,3}(x,y)=K((x_1,y_2,y_3),(y_1,x_2,x_3)).
\end{equation*}
We assume that each adjoint kernel satisfies the same estimates as the kernel $K$.

We then define the partial kernels.
Related to every interval $I^1$ there exists a kernel
$$
	K_{I^1} \colon (\R^2 \times \R^2) \setminus \{(x_{2,3}, y_{2,3}) \colon |x_2-y_2||x_3-y_3|=0\} \to \C,
$$
so that if $f_{23}, g_{23} \colon \R^2 \to \C$ have disjoint supports, then
\begin{equation*}
	\langle T(1_{I^1} \otimes f_{23}), 1_{I^1} \otimes g_{23}\rangle
	= \iint K_{I^1}(x_{23}, y_{23}) f_{23}(y_{23})g_{23}(x_{23}) \ud y_{23} \ud x_{23}.
\end{equation*}
The kernel $K_{I^1}$ is assumed to satisfy the following estimates.
Let $x_{23}, y_{23}$ be such that $|x_2-y_2||x_3-y_3|\not=0$. First, we assume the size estimate
\begin{equation*}
	|K_{I^1}(x_{23}, y_{23}) |
	\lesssim \frac{|I^1|}{|x_2-y_2||x_3-y_3|} D_{\theta}(|I^1|, x_2-y_2,x_3-y_3).
\end{equation*}
See again the beginning of Section \ref{sec:awf} for the definition of $D_{\theta}$.
Let then $x'_{23}=(x_2',x_3')$ be such that $|x_i'-x_i| \le |x_i-y_i|/2$ for $i=2,3$.
We assume the H\"older estimate
\begin{equation*}
	\begin{split}
		 & |K_{I^1}(x_{23}',y_{23})-K_{I^1}(x_{23},y_{23})| \\
		 & \hspace{1cm}\lesssim
		\Big( \frac{|x_2'-x_2|}{|x_2-y_2|}+\frac{|x_3'-x_3|}{|x_3-y_3|} \Big)^{\alpha_{23}}\frac{|I^1|}{|x_2-y_2||x_3-y_3|}
		D_{\theta}(|I^1|, x_2-y_2,x_3-y_3).
	\end{split}
\end{equation*}
We assume that the adjoint kernel defined by $K_{I^1}^*(x_{2,3}, y_{2,3})
	=K_{I^1}(y_{2,3}, x_{2,3})$ satisfies the same estimates.

Now, we briefly consider the other partial kernel representation.
Related to every rectangle $I^{23}$ there exists a classical Calder\'on-Zygmund kernel $K_{I^{23}}$
with H\"older exponent $\alpha_{1}$ and kernel constant $\lesssim |I^{2,3}|$
so that if $f_1, g_1 \colon \R \to \C$ have disjoint supports, then
\begin{equation}\label{E:eq68}
	\langle T(f_1 \otimes 1_{I^{23}}), g_1 \otimes 1_{I^{23}} \rangle
	= \iint K_{I^{23}}(x_1,y_1) f_1(y_1) g_1(x_1) \ud y_1 \ud x_1.
\end{equation}

See \cite{HLMV} to see that all previously known Zygmund type
SIOs/CZOs fall into this general framework.
We do not need all of these kernel estimates below. In addition, the CZO
assumption could here be replaced by assuming a priori boundedness.
\begin{thm}\label{thm:suffbddoff}
	Let $T$ be a Zygmund CZO (see \cite{HLMV}) with a kernel $K$ like above
	and the associated constant $\theta \in (0,1]$.
	Assume that $1 < p < q < \infty$ and
	$$
		\alpha := \frac{1}{p} - \frac{1}{q} \le \theta.
	$$
	Then we have
	$$
		\|[b,T]\|_{L^p \to L^q} \lesssim \|b\|_{\bmo_Z^{\alpha}}.
	$$
\end{thm}

\begin{proof}
	We begin by proving that
	\begin{equation}\label{eq:comrep}
		|[b, T](f)(z)|\le \int_{\R^3}  |b(z)-b(y)| |K(z,y)| |f(y)|\ud y
	\end{equation}
	for a.e. $z$. We do this for dense $f$ of the form
	$$
		f= \sum_{I\in \mathcal Q}c_I 1_I=\sum_{I\in \mathcal Q}c_I 1_{I^1\times I^2 \times I^3},
	$$
	where $\mathcal Q$ is some finite collection of disjoint rectangles.
	Notice that \eqref{eq:comrep} is likely an equality, but e.g. the analogous results
	for bi-commutators in \cite{AHLMO2021} use special properties of one-parameter
	CZOs that are not available to us here. So we settle for an inequality as that
	is all that we need.

	Define
	$$
		V(x) := \int_{\R^3}  |b(x)-b(y)| |K(x,y)| |f(y)|\ud y.
	$$
	This is a well-defined converging integral -- see e.g. the main
	calculation \eqref{eq:comest} later. Now, to prove \eqref{eq:comrep}
	fix $z=(z_1, z_2, z_3)$ that is a (rectangular) Lebesgue point of $[b,T]f$ and $Vf$.
	Moreover, we may assume $z_1\notin \cup_{I\in \mathcal Q}\partial I^1$.

	We have
	\begin{align*}
		[b, T](f)(z)=\lim_{\eps \to 0}\bla [b, T](f)\bra_{R(z, \eps)},
	\end{align*}
	where $R(z, \eps)=I(z_1, \eps)\times I(z_2, \eps)\times I(z_3, \eps^2)$
	is a Zygmund rectangle shrinking to $\{z\}$; here $I(z_j, t)$
	stands for an interval centered at $z_j\in \R$ of side length $t$.
	Actually, it is not particularly important that we use a Zygmund rectangle here,
	a cube would also work (but some estimates are perhaps slightly more aesthetic like this).
	In preparation for using the full kernel representation of $T$
	we define that
	\[
		R^c(z, \eps) := I(z_1, \eps)^c \times (I(z_2, \eps)\times I(z_3, \eps^2))^c.
	\]
	We decompose
	\begin{align*}
		\bla [b, T](f)\bra_{R(z, \eps)} & = \bla [b, T](f1_{R^c(z, \eps)})\bra_{R(z, \eps)}
		+ \bla (b- b(z))T(f 1_{\R^3\setminus R^c(z, \eps)})\bra_{R(z, \eps)}                                          \\
		                                & -\bla T\big((b- b(z))f 1_{\R^3\setminus R^c(z, \eps)}\big)\bra_{R(z, \eps)} \\
		                                & =: I_{\eps}+II_{\eps}+III_{\eps}.
	\end{align*}
	Notice that for the main term $I_{\eps}$
	we get using the full kernel representation that
	\begin{align*}
		|I_{\eps}| & =\Big|\Big\langle [b, T](f1_{R^c(z, \eps)}), \frac{1_{R(z, \eps)}}{|R(z, \eps)|}\Big\rangle\Big|             \\
		           & =  \Big|\Big\langle x \mapsto \int_{R^c(z, \eps)}  (b(x)-b(y)) K(x,y) f(y)\ud y\Big\rangle_{R(z, \eps)}\Big|
		\le \langle V \rangle_{R(z, \eps)}.
	\end{align*}
	We will soon show that
	\begin{equation}\label{eq:reduction}
		\lim\limits_{\eps\to 0} |II_{\eps}|+\lim\limits_{\eps\to 0}|III_{\eps}| =0.
	\end{equation}
	This then implies the claim -- we have
	$$
		|[b,T](f)(z)| \le \lim_{\eps \to 0} \langle V \rangle_{R(z, \eps)} = V(z).
	$$

	So to prove \eqref{eq:comrep} it remains to prove \eqref{eq:reduction}.
	Clearly, for $x\in R(z, \eps)$, we have $|b(x)-b(z)|\lesssim  \|b\|_{\bmo_Z^\alpha} \eps^{4\alpha}$ and hence
	$|II_{\eps}|$ is dominated by
	\begin{align*}
		\bla\big|(b- b(z))T(f 1_{\R^3\setminus R^c(z, \eps)})\big|\bra_{R(z, \eps)}
		 & \lesssim \|b\|_{\bmo_Z^\alpha} \eps^{4\alpha} \bla\big|T(f 1_{\R^3\setminus R^c(z, \eps)})\big|\bra_{R(z, \eps)}     \\
		 & \le  \|b\|_{\bmo_Z^\alpha} \eps^{4\alpha} \bla\big|T(f 1_{\R^3\setminus R^c(z, \eps)})\big|^s\bra_{R(z, \eps)}^{1/s} \\
		 & \lesssim  \|b\|_{\bmo_Z^\alpha} \eps^{4(\alpha-\frac 1s)}  \|f\|_{L^s}.
	\end{align*}
	By taking $s$ large enough (so that $1/s<\alpha$) we get $\lim\limits_{\eps\to 0}|II_{\eps}|=0$
	as desired.

	The estimate of $|III_{\eps}|$ is more involved.
	Since the collection $\mathcal Q$ is finite, it suffices to prove \eqref{eq:reduction} for $f=1_I$
	-- this is convenient for some partial kernel estimates below.
	Notice first that
	\[
		f 1_{\R^3\setminus R^c(z, \eps)}= f 1_{\R\times I(z_2, \eps)\times I(z_3, \eps^2)}
		+ f1_{I(z_1, \eps) \times (I(z_2, \eps)\times I(z_3, \eps^2))^c}.
	\]
	The part of $III_{\eps}$ related to the first term from above is easy -- indeed, we have
	\begin{align*}
		\big|\bla T\big((b- b(z)) & f 1_{\R\times I(z_2, \eps)\times I(z_3, \eps^2)} \big)\bra_{R(z, \eps)}\big|     \\
		                          & \le \bla |T\big((b- b(z))f
		1_{\R\times I(z_2, \eps)\times I(z_3, \eps^2)} \big)|^s\bra_{R(z, \eps)}^{1/s}                               \\
		                          & \lesssim \eps^{-4/s}\| (b- b(z))f 1_{\R\times I(z_2, \eps)\times I(z_3, \eps^2)}
		\|_{L^s}\lesssim  \|b\|_{\bmo_Z^\alpha} \eps^{4(\alpha-\frac 1s)}  \|f\|_{L^s},
	\end{align*}
	and so we are reduced to proving
	\begin{equation}\label{eq:1103}
		\lim\limits_{\eps\to 0} \big|\bla T\big((b- b(z))f 1_{I(z_1, \eps) \times (I(z_2, \eps)\times I(z_3, \eps^2))^c}\big)\bra_{R(z, \eps)}\big|=0.
	\end{equation}

	If $z_1 \notin I^1$, then by our assumption $z_1 \notin \bar{I^1}$.
	So if $\eps$ is sufficiently small we have $I(z_1, \eps) \cap I^1=\emptyset$
	and so
	\[
		f1_{I(z_1, \eps) \times (I(z_2, \eps)\times I(z_3, \eps^2))^c}=0
	\]
	and the desired estimate is trivial. Now, if $z_1 \in I^1$, we have
	$I(z_1, \eps) \subset I^1$ for all sufficiently small $\eps$.
	Then we are able to use the partial kernel representation:
	\begin{align}
		 & \bla T\big((b- b(z))f 1_{I(z_1, \eps) \times (I(z_2, \eps)\times I(z_3, \eps^2))^c}
		\big)\bra_{R(z, \eps)} \label{eq:pa}                                                   \\
		 & = \eps^{-4}\bla   T\big( 1_{I(z_1, \eps)} \otimes (b- b(z))
		1_{ I^{23}\setminus(I(z_2, \eps)\times I(z_3, \eps^2))} \big),
		1_{I(z_1, \eps)} \otimes   1_{ I(z_2, \eps)\times I(z_3, \eps^2)} \bra\nonumber        \\
		 & =\eps^{-4}\int_{ I(z_2, \eps)\times I(z_3, \eps^2)}
		\int_{I^{23}\setminus (I(z_2, \eps)\times I(z_3, \eps^2))}
		(b(y_3)- b(z_3)) K_{I(z_1, \eps)}(x_{23}, y_{23})\ud y_{23}\ud x_{23},\nonumber
	\end{align}
	where we have abused the notation so that $b(x_3):=b(x_1, x_2, x_3)$
	for whatever $x_1, x_2$.

	We then split the RHS of \eqref{eq:pa} as
	\begin{align*}
		 & \eps^{-4}\int_{ I(z_2, \eps)\times I(z_3, \eps^2)}\int_{\substack{I^{23}\setminus (I(z_2, \eps)\times I(z_3, \eps^2)) \\ |y_3-z_3|\le 2\sqrt \eps}} +\eps^{-4}\int_{ I(z_2, \eps)\times I(z_3, \eps^2)}\int_{\substack{I^{23}\setminus (I(z_2, \eps)\times I(z_3, \eps^2))\\ |y_3-z_3|> 2\sqrt \eps}}\\
		 & =: III_1+III_2.
	\end{align*}
	We estimate $III_1$ first using that there
	\[
		|b(y_3)- b(z_3)|\lesssim_{\|b\|_{\bmo_Z^{\alpha}}} \eps^{\alpha}.
	\]
	Using this and the size estimate of the partial kernel,
	\begin{align*}
		|K_{I(z_1, \eps)}(x_{23}, y_{23})|
		\lesssim \frac{\eps}{ |x_2-y_2||x_3-y_3| } \Big( \frac{\eps |x_2-y_2|}{|x_3-y_3|}
		+\frac{|x_3-y_3|}{\eps |x_2-y_2|} \Big)^{-\theta},
	\end{align*}
	we can dominate $|III_1|$ by
	\begin{align*}
		 & \eps^{\alpha-4}\int_{ I(z_2, \eps)\times I(z_3, \eps^2)}
		\int_{\substack{I^{23}\setminus (I(z_2, \eps)\times I(z_3, \eps^2))
		\\ |x_2-y_2|>\eps^{-1}|x_3-y_3|}}  \frac{\eps^{1-\theta}}{|x_2-y_2|^{1+\theta}
		|x_3-y_3|^{1-\theta}}\ud y_{23}\ud x_{23}                           \\
		 & \quad+ \eps^{\alpha-4}\int_{ I(z_2, \eps)\times I(z_3, \eps^2)}
		\int_{\substack{I^{23}\setminus (I(z_2, \eps)\times I(z_3, \eps^2)) \\
				|x_2-y_2|\le\eps^{-1}|x_3-y_3|}}  \frac{\eps^{1+\theta}}
		{|x_2-y_2|^{1-\theta} |x_3-y_3|^{1+\theta}}\ud y_{23}\ud x_{23}.
	\end{align*}
	Next, we note that
	\[
		I^{23}\setminus (I(z_2, \eps)\times I(z_3, \eps^2))
		= \Big[(I^2 \setminus I(z_2, \eps)) \times I^3\Big] \cup \Big[I^2 \times (I^3\setminus I(z_3, \eps^2))\Big].
	\]
	To demonstrate the idea of how to handle the terms of $III_1$
	we estimate one part of it as follows:
	\begin{align*}
		 & \eps^{\alpha-4}\int_{ I(z_2, \eps)\times I(z_3, \eps^2)}
		\int_{\substack{ (I^2 \setminus I(z_2, \eps))\times I^3                               \\ |x_2-y_2|>\eps^{-1}|x_3-y_3|}}  \frac{\eps^{1-\theta}}{|x_2-y_2|^{1+\theta} |x_3-y_3|^{1-\theta}}\ud y_{23}\ud x_{23}\\
		 & \lesssim  \eps^{\alpha-4}\int_{ I(z_2, \eps)\times I(z_3, \eps^2)}
		\int_{I^2 \setminus I(z_2, \eps)}\frac \eps{|x_2-y_2|}\ud y_2 \ud x_{23}              \\
		 & =\eps^{\alpha-1} \int_{ I(z_2, \eps)}\int_{I^2 \setminus I(z_2, \eps)}
		\frac 1{|x_2-y_2|}\ud y_2 \ud x_{2}                                                   \\
		 & = \eps^{\alpha-1}\big\langle |H(1_{I(z_2, \eps)})|, 1_{I^2 \setminus I(z_2, \eps)}
		\big\rangle\lesssim  \eps^{\alpha-1}\eps^{1-1/s} |I^2|^{1/s},
	\end{align*}
	where $H$ is the Hilbert transform. This tends to $0$ as $\eps \to 0$ by taking $1/s<\alpha$.
	Other terms that come from $III_1$ can be handled in a similar way.

	Finally, we still have to estimate $III_2$. Assuming that $\eps<1$ we have
	\[
		|x_3-y_3|\ge |y_3-z_3|- |x_3-z_3|> 2\sqrt \eps - \eps^2> \sqrt \eps.
	\]
	Notice also that
	\[
		|x_2-y_2|\le |x_2-z_2|+|z_2-y_2|< \eps+ (\dist(z_2, I^2)+|I^2|)
		< 1 + \dist(z_2, I^2)+|I^2| := L.
	\]
	Here the philosophical reason why in $III_2$ we can always afford to estimate
	\begin{align*}
		|K_{I(z_1, \eps)}(x_{23}, y_{23})|
		 & \lesssim \frac{\eps}{ |x_2-y_2||x_3-y_3| } \Big( \frac{\eps |x_2-y_2|}{|x_3-y_3|}
		+\frac{|x_3-y_3|}{\eps |x_2-y_2|} \Big)^{-\theta}                                    \\
		 & \lesssim \frac{\eps^{1+\theta}}{|x_2-y_2|^{1-\theta}|x_3-y_3|^{1+\theta}}
	\end{align*}
	is that if $\eps$ is sufficiently small so that $\sqrt{\eps} L < 1$ we
	automatically have
	\[
		|x_3-y_3|> \sqrt \eps> \eps L \ge \eps |x_2-y_2|.
	\]
	Using $|b(y_3)-b(z_3)|\le \sup\limits_{y_3\in I^3}|b(y_3)|+|b(z_3)| \lesssim 1$, we have
	\begin{align*}
		 & |III_2|
		\lesssim \eps^{-4} \int_{ I(z_2, \eps)\times I(z_3, \eps^2)}
		\int_{\substack{|x_2-y_2|< L                                                          \\
				|x_3-y_3|>\sqrt \eps}}
		\frac{\eps^{1+\theta}}{|x_2-y_2|^{1-\theta} |x_3-y_3|^{1+\theta}}\ud y_{23}\ud x_{23} \\
		 & \lesssim \eps^{-3+\theta} L^\theta
		\int_{ I(z_2, \eps)\times I(z_3, \eps^2)}\int_{|x_3-y_3|>\sqrt \eps}
		\frac 1{|x_3-y_3|^{1+\theta}}\ud y_3 \ud x_{23}                                       \\
		 & \lesssim \eps^{-3+\theta} \eps^3 \eps^{-\theta/2} = \eps^{\theta/2},
	\end{align*}
	which tends to $0$ as $\eps \to 0$.
	This completes the proof of \eqref{eq:1103} and so
	we are done with the proof of \eqref{eq:comrep}.

	We will then estimate the RHS of \eqref{eq:comrep}. First, a few elementary inequalities.
	Let $t_i > 0$ -- we notice that
	\begin{align*}
		\frac{t_3^{2\alpha}}{t_1^{1+\theta}t_2^{1+\theta}t_3^{1-\theta}}
		= \frac{1}{\prod_i t_i^{1-\alpha}} \Big(\frac{t_3}{t_1t_2} \Big)^{\alpha+\theta}
		\le \frac{1}{\prod_i t_i^{1-\alpha}}
	\end{align*}
	if $t_3 \le t_1t_2$. A similar (but not the same) inequality in the case $t_1t_2 \le t_3$ requires
	our assumption $\alpha \le \theta$ from the statement of theorem. Indeed, we have
	\begin{align*}
		\frac{t_3^{2\alpha}}{t_1^{1-\theta}t_2^{1-\theta}t_3^{1+\theta}}
		= \frac{1}{\prod_i t_i^{1-\alpha}} \Big(\frac{t_1t_2}{t_3} \Big)^{\theta-\alpha}
		\le \frac{1}{\prod_i t_i^{1-\alpha}}.
	\end{align*}
	For convenience, assume that $\|b\|_{\bmo_Z^{\alpha}} \le 1$ (or just absorb this
	to the constant in $\lesssim$).
	Using the size estimate of $K$ and the above inequalities we then have
	\begin{equation}\label{eq:comest}
		\begin{split}
			\int_{\R^3} & |b(x)-b(y)| |K(x,y)| |f(y)|\ud y                                                         \\
			            & \lesssim \int_{\R^3} \frac{|x_3-y_3|^{2\alpha}}
			{\big(|x_1-y_1||x_2-y_2||x_3-y_3|\big)^{1-\theta}
			(|x_1-y_1|^2|x_2-y_2|^2+|x_3-y_3|^2)^\theta }|f(y)|\ud y                                               \\
			            & \le  \int_{\R^3} \frac{|x_3-y_3|^{2\alpha}
			1_{\{|x_1-y_1||x_2-y_2|\ge |x_3-y_3|\}}}
			{|x_1-y_1|^{1+\theta}|x_2-y_2|^{1+\theta}|x_3-y_3|^{1-\theta}}|f(y)|\ud y                              \\
			            & \hspace{3cm}+ \int_{\R^3} \frac{|x_3-y_3|^{2\alpha}
			1_{\{|x_1-y_1||x_2-y_2| \le |x_3-y_3|\}}}
			{|x_1-y_1|^{1-\theta}|x_2-y_2|^{1-\theta}|x_3-y_3|^{1+\theta}}|f(y)|\ud y                              \\
			            & \le  \int_{\R^3}\frac{1}{  |x_1-y_1|^{1-\alpha}|x_2-y_2|^{1-\alpha}|x_3-y_3|^{1-\alpha}}
			|f(y)|\ud y= (I_\alpha^1I_\alpha^2I_\alpha^3)(f)(x),
		\end{split}
	\end{equation}
	where $I_{\alpha}^i$ is the Riesz potential on the corresponding parameter.

	The conclusion follows from the $L^p\to L^q$ boundedness of $I_\alpha^1  I_\alpha^2  I_\alpha^3$,
	which is the following argument using the $L^p \to L^q$ boundedness of the individual
	Riesz potentials combined with Minkowski's integral inequality (and the non-negativity of the Riesz potential):
	\begin{align*}
		\| I_\alpha^1  I_\alpha^2  I_\alpha^3 (f)\|_{L^q}
		 & \lesssim \big\| \|   I_\alpha^2  I_\alpha^3 (f)\|_{L^p_{x_1}}
		\big\|_{L^q_{x_{23}}}\le \big\| \|   I_\alpha^2  I_\alpha^3 (f)\|_{L^q_{x_{23}}}\big\|_{L^p_{x_{1}}} \\
		 & \lesssim   \Big\| \big\| \| I_\alpha^3 (f)\|_{L^p_{x_{2}}}
		\big\|_{L^q_{x_3}} \Big\|_{L^p_{x_{1}}}
		\le  \big\| \|I_\alpha^3 (f)\|_{L^q_{x_{3}}}\big\|_{L^p_{x_{12}}}
		\lesssim \|f\|_{L^p}.
	\end{align*}
\end{proof}
\begin{rem}
	Notice that if $\theta\ge 1/2$, then we can remove the restriction $\alpha \le \theta$.
	Indeed, if $b \in \bmo_Z^{\alpha}$ for $\alpha > 1/2$, $b$ is a constant and the result is trivial.
\end{rem}

\section{Necessity results for compactness}
In the compactness context we assume that we are given a Zygmund SIO $T$
with a non-degenerate kernel $K$ -- this is because off-diagonal constants
do not appear as natural in the compactness context and, instead, we want to talk
about the compactness of the actual commutator $[b,T]$. But the kernel itself can
be weak as in Section \ref{sec:awf}, since we only rely on AWF.

Interestingly, it turns out that compactness implies that $b$ is constant.
An easier version of this Zygmund argument that we present in this chapter gives the same result also for $[b,T]$, where
$T$ is a bi-parameter SIO. See the Introduction for the significance of this.

We need the following in our compactness necessity proofs.
The proof is similar to that in Martikainen--Oikari \cite{Martikainen-Oikari-2024}.
\begin{lem}\label{lem:selection}
	Suppose $(R_i)$, $R_i = I_i^1 \times I_i^2 \times I_i^3$,
	is a sequence of Zygmund rectangles so that for \textbf{some} $j \in \{1,2,3\}$
	we have $\ell(I_i^j) \to 0$
	as $i \to \infty$. Then there is a subsequence
	$(R_{i(k)})$ so that either $(R_{i(k)})$ or $(\wt R_{i(k)})$ is pairwise disjoint.
	Here the reflected rectangles are defined using some underlying non-degenerate $K$.
\end{lem}
\begin{proof}
	The fact that these are Zygmund rectangles does not really play a role in this proof -- other than
	for the fact that these are the rectangles for which we have defined reflection (using some
	non-degenerate kernel $K$ that does not otherwise appear in the proof).

	Notice first that if $\dist(I_i^j, 0) \to \infty$,
	then it is obvious how to choose a subsequence so that $(I^j_{i(k)})$
	is pairwise disjoint. Indeed, suppose that we have picked $i(1) < \cdots < i(N)$
	so that $\{I^j_{i(k)}\}_{k=1}^N$ is pairwise disjoint.
	Let $M>0$ be such that $\cup_{k=1}^N I^j_{i(k)}\subset [-M, M]$.
	Then here exists $i(N+1) > i(N)$ so that
	$\dist(I^j_{i(N+1)},0) > M$. Now, obviously
	$\{I^j_{i(k)}\}_{k=1}^{N+1}$ is pairwise disjoint, and we can continue by induction.

	Thus, during the proof of the theorem we may assume
	that $\dist(I_i^j, 0) \to \infty$ does not hold.
	Passing to a subsequence we may then assume $\dist(I_i^j, 0) \lesssim 1$
	for all $i$.

	So assume now that $\ell(I_i^j) \to 0$. With the above reduction in place this
	allows us to choose a closed dyadic interval $Q_0$ centered at $0$ so that $\bigcup_i I_i^j \subset Q_0$.
	There must exist $Q_ 1\in \ch(Q_0)$ so that
	$F(Q_1) :=	\{ i \colon I_i^j \cap Q_1\not=\emptyset \}$
	has infinitely many elements. Similarly, there must exist $Q_2 \in \ch(Q_1)$
	so that $F(Q_2) = \{ i \colon I_i^j \cap Q_2\not=\emptyset \} \subset F(Q_1)$
	has infinitely many elements. Continue like this to obtain closed
	intervals $Q_0 \supsetneq Q_1 \supsetneq Q_2 \supsetneq \cdots$ with $\ell(Q_k) \to 0$,
	and the corresponding infinite index sets $F(Q_k) \supset F(Q_{k+1})$.
	Pick $i(1) \in F(Q_1)$. Then, using that $F(Q_2)$ is infinite, pick
	$i(2) \in F(Q_2)$ with $i(2) > i(1)$. Continue like this to obtain
	indices $i(1) < i(2) < \cdots$ so that $i(k) \in F(Q_k)$, i.e.,
	$I_{i(k)}^j \cap Q_k \ne \emptyset$. Now, re-index so that
	$I_i^j \cap Q_i \ne \emptyset$ for all $i$.

	Let $x$ be such that
	$$
		\{x\} = \bigcap_i Q_i.
	$$
	Notice that the intersection is non-empty as each $Q_i \supset Q_{i+1}$ is compact, and then a
	singleton as $\ell(Q_i) \to 0$. Let
	$$
		\calF := \{i \colon x \not \in I_i^j\}.
	$$
	We assume first that $\#\calF = \infty$. Notice that $I_i^j \cap Q_i \ne \emptyset$
	and $x \in Q_i$ implies that
	$$
		\dist(x, I_i^j) \le \ell(Q_i) \to 0.
	$$
	As also $\ell(I^j_i) \to 0$ we have
	$$
		\dist(x, I_i^j) + \ell(I_i^j) \to 0.
	$$
	Pick $i(1) \in \calF$. Suppose $i(1) < i(2) < \cdots < i(N)$, $i(k) \in \calF$, have
	already been chosen so that the intervals $I^j_{i(1)}, \ldots, I^j_{i(N)}$ are pairwise disjoint.
	Then pick -- using that $\calF$ is infinite -- $i(N+1) \in \calF$, $i(N+1) > i(N)$, so that
	$$
		\dist(I_{i(N+1)}^j, x) + \ell(I_{i(N+1)}^j) < \dist\Big(x, \bigcup_{k=1}^N I_{i(k)}^j\Big)/10,
	$$
	where it is important to notice that the RHS is strictly positive as $x$
	does not belong to any of these finitely many intervals.
	It is now clear that $(I_{i(k)}^j)_{k=1}^{N+1}$ is a disjoint collection of cubes.
	Induction finishes the case $\#\calF = \infty$.

	Suppose then $\#\calF < \infty$. Since
	$$
		\N \setminus \calF = \{i \colon x \in I_i^j \} \subset \{i \colon x \not \in \wt I_i^j\} =: \wt \calF,
	$$
	we must have that $\# \wt \calF = \infty$. Still $\ell(\wt I_i^j) = \ell(I_i^j) \to 0$.
	To see that also $\dist(x, \wt I_i^j) \to 0$ choose $y \in \wt I_i^j$ so that $\dist(x, \wt I_i^j)
		= |y-x|$ and $z \in I_i^j$ so that $\dist(x, I_i^j) = |x-z|$ and estimate
	$$
		\dist(x, \wt I_i^j) = |y-x| \le |y-z| + \dist(x, I_i^j)
		\lesssim \ell(I_i^j) + \ell(Q_i) \to 0.
	$$
	After this observation we can repeat the proof from the case $\# \calF = \infty$
	to produce a subsequence so that $(\wt I_{i(k)}^j)$ is pairwise disjoint.
\end{proof}

We derive the following using AWF and the previous lemma.
\begin{thm}
	Let $T$ be a non-degenerate Zygmund invariant SIO,
	$b \in L^1_{\loc}$, $1 < p \le q < \infty$ and $\alpha := 1/p-1/q \ge 0$.
	Assume that the commutator
	$[b, T] \colon L^p \to L^q$ is compact.
	Then there cannot exist $c > 0$ and a sequence of Zygmund rectangles $(R_i)$,
	$R_i = I_i^1 \times I_i^2 \times I_i^3$, so that
	$\calO_{\alpha}(b, R_i) \ge c$ for all $i$ and that for some $j \in \{1,2,3\}$ we have
	$\ell(I_i^j) \to 0$.
\end{thm}
\begin{proof}
	Suppose, aiming for a contradiction, that such a sequence of rectangles exists.
	Using the above lemma and passing to a subsequence we can assume that, in addition,
	$(R_i)$ or $(\wt R_i)$ is pairwise disjoint.
	By AWF, Lemma \ref{lem:AWFCorBdd}, we find functions
	$$
		\varphi_{i,1}:=\varphi_{R_i, 1}, \, \varphi_{i,2}:=\varphi_{R_i,2}, \,
		\psi_{i,1}:=\psi_{\wt R_i, 1}, \, \psi_{i,2}:=\psi_{\wt R_i, 2}
	$$
	so that for $j = 1,2$ we have
	\[
		|\varphi_{i,j}|\lesssim 1_{R_i}, \qquad |\psi_{i,j}|\lesssim |R_i|^{-1}1_{\wt R_i}
	\]
	and
	\begin{align*}
		\osc(b, R_i) \lesssim \sum_{j=1}^2 |\langle [b,T]\varphi_{i,j}, \psi_{i,j}\rangle|
		 & = |R_i|^{\alpha} \sum_{j=1}^2 |\langle [b,T](|R_i|^{-1/p}\varphi_{i,j}), |R_i|^{1/q}\psi_{i,j}\rangle| \\
		 & =: |R_i|^{\alpha} \sum_{j=1}^2 |\langle [b,T]\wt \varphi_{i,j}, \wt \psi_{i,j}\rangle|.
	\end{align*}
	Notice that with these normalizations we have
	$$
		\|\wt \varphi_{i,j}\|_{L^p} \lesssim 1 \qquad \textup{and} \qquad
		\|\wt \psi_{i,j}\|_{L^{q'}} \lesssim 1.
	$$

	Suppose now, for instance, that $(R_i)$ is pairwise disjoint.
	Using the compactness of $[b, T]$ we can, after passing to a subsequence,
	assume that $ [b,T]\wt \varphi_{i,j} \to \Phi_j$ in $L^q$ for
	some $\Phi_j\in L^q$ simultaneously for $j=1,2$.
	Then we have
	\begin{align*}
		c\le \calO_{\alpha}(b, R_i)
		 & \lesssim \sum_{j=1}^2 |\langle [b,T]\wt \varphi_{i,j}, \wt \psi_{i,j}\rangle|                       \\
		 & \lesssim \sum_{j=1}^2 \big\| [b,T]\wt \varphi_{i,j}\big\|_{L^q} \to \sum_{j=1}^2  \|\Phi_j\|_{L^q}.
	\end{align*}
	This means that either $\Phi_1$ or $\Phi_2$ must be non-trivial.
	But this contradicts Lemma \ref{lem:seq} below. The case that  $(\wt R_i)$  is pairwise disjoint is similar
	-- just apply the previous proof to
	the compact operator $[b, T]^* \colon L^{q'} \to L^{p'}$. We are done.
\end{proof}

The following is a well-known Uchiyama type lemma that we used above
to reach the contradiction.
\begin{lem}\label{lem:seq}
	Let $U\colon L^{p}\to L^{q}$ be a bounded linear operator for some fixed $1<p,q<\infty.$
	Then, there does \textbf{not} exist a
	sequence $\{u_k\}_{k=1}^\infty\subset L^{p}$ with the following properties:
	\begin{enumerate}[$(i)$]
		\item $\|u_k\|_{L^{p}}\lesssim 1$ uniformly on $k$;
		\item
		      \begin{align*}
			      \supp u_k \cap \supp u_m = \emptyset,\qquad k\not=m;
		      \end{align*}
		\item
		      $$
			      Uu_k \to \Phi
		      $$
		      in $L^{q}$ for $\Phi$ with $\|\Phi\|_{L^q} \ne 0$.
	\end{enumerate}
\end{lem}
\begin{proof}
	This is well-known, see \cites{HOS2023, Uch1978}, but so short that we give the proof for completeness.
	Assume such $(u_k)$ exists. Without loss of generality
	$$
		\|Uu_k - \Phi\|_{L^q} \le b_k
	$$
	for all $k$, where $(b_k) \in \ell^{p'}$. Choose then $a_k \ge 0$ so that $(a_k) \in \ell^p \setminus
		\ell^1$. Notice that
	$$
		\sum_{k=1}^N a_k \sim \sum_{k=1}^N a_k\|\Phi\|_{L^q}
		\le \Big\| \sum_{k=1}^N a_k Uu_k - \sum_{k=1}^N a_k\Phi\Big\|_{L^q}
		+ \Big\| U\Big(\sum_{k=1}^N a_k u_k)\Big\|_{L^q} =: I + II.
	$$
	Here
	$$
		I \le \sum_{k=1}^N a_k \|Uu_k - \Phi\|_{L^q} \le \sum_{k=1}^{\infty} a_k b_k < \infty
	$$
	and
	$$
		II \lesssim \Big\| \sum_{k=1}^N a_k u_k \Big\|_{L^p} = \Big( \sum_{k=1}^N |a_k|^p \|u_k\|_{L^p}^p \Big)^{1/p}
		\lesssim \Big( \sum_{k=1}^{\infty} |a_k|^p \Big)^{1/p} < \infty,
	$$
	so $(a_k) \in \ell^1$ after all -- a contradiction.
\end{proof}

We can now derive the main corollary -- compactness actually forces $b$ to be a constant
both in the diagonal $L^p \to L^p$ and in the off-diagonal $L^p \to L^q$, $p < q$.
We do the diagonal case first.
\begin{thm}\label{thm:comdiag}
	Let $T$ be a non-degenerate Zygmund invariant SIO,
	$b \in L^1_{\loc}$ and $1 < p < \infty$. Assume that the commutator
	$[b, T] \colon L^p \to L^p$ is compact. Then $b$ is a constant.
\end{thm}
\begin{proof}
	Recall that we proved that if $\eps > 0$ then there exists $t > 0$
	so that for all Zygmund rectangles $R = I^1 \times I^2 \times I^3$
	satisfying that $\ell(I^j) < t$ for some $j = 1,2,3$,
	we have $\osc(b, R_i) < \eps$. This is part of what one could possibly
	call some kind of $\VMO$ space -- but it is not worth making such a definition as it
	turns out this actually already forces $b$ to be a constant.

	Indeed, for an interval $I^1 \subset \R$
	define $b_{I^1}:= |I^1|^{-1}\int_{I^1} |b- \langle b\rangle_{I^1}|\ud x_1$ -- a function of $x_{23}$.
	Then, for any $I^{23} = I^2 \times I^3$ we have
	\begin{align*}
		\langle b_{I^1}\rangle_{I^{23}} & =\frac1{|R|} \int_{R} |b- \langle b\rangle_{I^1}|\ud x   \\
		                                & \le  \frac1{|R|} \int_{R} |b- \langle b\rangle_{R}|\ud x
		+\frac1{|I^{23}|} \int_{I^{23}} |\langle b\rangle_{I^1}- \langle b\rangle_{R}|\ud x \le 2\osc(b, R).
	\end{align*}
	With a fixed $I^1$ and $x_{23}$ we consider rectangles $I^{23}$ such that
	$x_{23}\in I^{23}$ and $|I^3|=|I^1||I^2|$.
	By the rectangular Lebesgue differentiation theorem, there exists a Lebesgue measure zero set $N_{I^1} \subset \R^2$
	such that for all $x_{23} \notin N_{I^1}$ we have
	\[
		b_{I^1}(x_{23})=\lim_{|I^2|\to 0}\langle b_{I^1}\rangle_{I^{23}}\le 2 \lim_{|I^2|\to 0}\osc(b, R)=0.
	\]

	We have to dispose of the $I^1$ dependence in the exceptional set.
	To this end, let $\mathcal F_1$ be the collection of intervals in $\R$
	with rational centers and rational side lengths.
	We define
	\[
		N_{23}:= \bigcup_{I^1\in \mathcal F_1}N_{I^1}.
	\]
	Clearly $|N_{23}|=0$.
	Then for every $I^1 \in \calF_1$ we have
	\[
		b_{I^1}(x_{23})= \frac 1{|I^1|}\int_{I^1} |b(\cdot, x_{23})- \langle b(\cdot, x_{23})\rangle_{I^1}|\ud x_1=0,\quad x_{23}\notin N_{23}.
	\]
	Next, we have to dispose of the limitation $I^1 \in \calF_1$. So let $J^1 \subset \R$
	be an arbitrary interval. Choose $I^1\in \calF_1$ with $J^1\subset I^1\subset 2J^1$ and estimate,
	in the complement of the set $N_{23}$, as follows
	\begin{align*}
		\frac 1{|J^1|}\int_{J^1} |b- \langle b\rangle_{J^1}|\ud x_1 &
		\le\frac 1{|J^1|}\int_{J^1} |b- \langle b\rangle_{I^1}|\ud x_1
		+ | \langle b\rangle_{J^1}-\langle b\rangle_{I^1}|                                                                               \\
		                                                            & \le \frac 2{|J^1|}\int_{J^1} |b- \langle b\rangle_{I^1}|\ud x_1    \\
		                                                            & \le \frac 4{|I^1|}\int_{I^1} |b- \langle b\rangle_{I^1}|\ud x_1=0.
	\end{align*}
	So we have shown that for all intervals $I^1 \subset \R$ we have
	\begin{equation}\label{eq:vanish1}
		\frac 1{|I^1|}\int_{I^1} |b- \langle b\rangle_{I^1}|\ud x_1 = 0
	\end{equation}
	almost everywhere.
	By symmetry, we have for all intervals $I^2 \subset \R$ that
	\begin{equation}\label{eq:vanish2}
		\frac 1{|I^2|}\int_{I^2} |b- \langle b\rangle_{I^2}|\ud x_2=0
	\end{equation}
	almost everywhere.

	Finally, we now consider a Zygmund rectangle $R=I^1\times I^2 \times I^3$
	and will show that $\osc(b,R) = 0$ ending the proof.
	We define $I_j^1=2^{-j}I^1$, where $cI^1$ means the interval with the same
	center as $I^1$ and of side length $c\ell(I^1)$. Likewise, let $I_j^2=2^j I^2$.
	We define $R_j:=I_j^1\times I_j^2\times I^3$. As $R_j$ is Zygmund
	with $\ell(I_j^1) \to 0$ as $j \to \infty$, we have
	\begin{equation}\label{eq:vanish3}
		\lim_{j\to \infty}\frac 1{|R_j|} \int_{R_j} |b-\langle b\rangle_{R_j}|=0.
	\end{equation}
	Now, notice the following -- for any constant $c$ and for an arbitrary $j$ we have
	\begin{align*}
		\frac 1{|R_j|}\int_{R_j}|b-c|
		 & \le \frac 1{|R_j|}\int_{R_j}|b-\langle b\rangle_{I_{j+1}^1}|
		+ \frac 1{|R_j|}\int_{R_j}|\langle b\rangle_{I_{j+1}^1}- \langle b\rangle_{I_{j+1}^{1}\times I_{j+1}^{2}}| \\
		 & \qquad+ \frac 1{|R_j|}\int_{R_j}|
		\langle b\rangle_{I_{j+1}^{1}\times I_{j+1}^{2}}-c|
		=:A_1+A_2+A_3.
	\end{align*}
	We see, using \eqref{eq:vanish1}, that
	\begin{align*}
		A_1 & \le \frac 1{|R_j|}\int_{R_j}|b-\langle b\rangle_{I_{j}^1}|
		+\frac 1{|R_j|}\int_{R_j}|\langle b\rangle_{I_{j+1}^1}-\langle b\rangle_{I_{j}^1}|            \\
		    & =0+ \frac 1{|R_j|}\int_{R_j}|\langle (b-\langle b\rangle_{I_{j}^1})\rangle_{I_{j+1}^1}|
		\lesssim  \frac {1}{|R_j|}\int_{R_j}|b-\langle b\rangle_{I_{j}^1}|=0.
	\end{align*}
	Next, using \eqref{eq:vanish2}, we also have that
	\begin{align*}
		A_2\le \frac 1{|R_j|}\int_{R_j} \frac 1{|I_{j+1}^1|}\int_{I_{j+1}^1}|b- \langle b\rangle_{I_{j+1}^2}|
		\lesssim \frac {1}{|R_{j+1}|}\int_{R_{j+1}} |b- \langle b\rangle_{I_{j+1}^2}|=0.
	\end{align*}
	For $A_3$, there holds that
	\begin{align*}
		A_3= \frac 1{|I^3|}\int_{I^3}|  \langle b\rangle_{I_{j+1}^{1}\times I_{j+1}^{2}}-c|
		\le \frac 1{|R_{j+1}|}\int_{R_{j+1}} |b-c|.
	\end{align*}
	Hence, we have proved that for all $c, j$ we have
	\[
		\frac 1{|R_j|}\int_{R_j}|b-c|\le  \frac 1{|R_{j+1}|}\int_{R_{j+1}} |b-c|.
	\]
	Since $R=R_0$, we can use the above estimate repeatedly to get that for all $j$ we have
	\begin{align*}
		\frac 1{|R|}\int_{R}|b-\langle b\rangle_{R}|
		 & \le  \frac 1{|R_{j}|}\int_{R_{j}} |b- \langle b\rangle_{R}|   \\
		 & \le \frac 1{|R_{j}|}\int_{R_{j}} |b- \langle b\rangle_{R_j}|
		+ |\langle b \rangle_{R_j} - \langle b \rangle_R|                \\
		 & \le \frac 1{|R_{j}|}\int_{R_{j}} |b- \langle b\rangle_{R_j}|
		+ \frac 1{|R|}\int_{R} |b- \langle b\rangle_{R_j}|
		 & \le \frac 2{|R_{j}|}\int_{R_{j}} |b- \langle b\rangle_{R_j}|.
	\end{align*}
	By \eqref{eq:vanish3} it remains to let $j \to \infty$.
\end{proof}

\begin{thm}\label{thm:compactnessoff}
	Let $T$ be a non-degenerate Zygmund invariant SIO,
	$b \in L^s_{\loc}$ for some $1<s<\infty$, and $1 < p <q< \infty$. Assume that the commutator
	$[b, T] \colon L^p \to L^q$ is compact. Then $b$ is a constant.
\end{thm}
\begin{proof}
	To simplify our life a bit, we use that by Theorem \ref{thm:equidefoff} we already know
	(as $[b,T]$ is, in particular, bounded $L^p \to L^q$) that
	\begin{equation}
		|b(x)-b(y)|\lesssim |x_3-y_3|^{2\alpha}.
	\end{equation}
	In particular, $b$ is constant in the first and second variables.
	So we need only to prove that given $x = (x_{12}, x_3)$ and $y_3 \ne x_3$ we have
	$b(x) = b(x_{12}, y_3)$.

	To this end, we recall that we proved that if $\eps > 0$ then there exists $t > 0$
	so that for all Zygmund rectangles $R = I^1 \times I^2 \times I^3$
	satisfying that $\ell(I^j) \le t$ for some $j = 1,2,3$,
	we have $\calO_{\alpha}(b, R_i) < \eps$, $\alpha := 1/p - 1/q > 0$.
	We will use this with $j = 1$ below.

	Let $I_k, J_k$ $(k\in \Z_+)$
	be the Zygmund rectangles centered at $x$ and $(x_{12}, y_3)$, respectively,
	that have side lengths $2^{-k}t$, $2^{-k}|x_3-y_3|/t$ and $2^{-2k}|x_3-y_3|$.
	We also let $I_0=J_0 $ be the Zygmund rectangle centered at $(x_1,x_2, (x_3+y_3)/2)$
	with side lengths $t, 2|x_3-y_3|/t$ and $2|x_3-y_3|$.
	Repeating the calculus in the proof of Theorem \ref{thm:equidefoff}
	we get
	\[
		|b(x)-b(x_{12}, y_3)|\lesssim \eps |x_3-y_3|^{2\alpha}.
	\]
	Since $\eps$ was arbitrary, we have $b(x)-b(x_{12}, y_3)=0$ as desired.
\end{proof}

\appendix

\section{A concrete Fefferman--Pipher multiplier with unbounded derivatives}\label{app:FP}
In this section we give an example of a Fefferman--Pipher multiplier $m$ that has the feature that
$\partial_{i} m \notin L^\infty$ for all $i=1,2,3$.
In general, we have not seen (but they may exist) explicit Fefferman--Pipher multipliers in the literature so this can be interesting on its own.
We originally came up with this example to indicate a problem in
the proof of Theorem 1.5 in \cite{DLOPW-ZYGMUND} (the authors have told us this preprint
is now withdrawn based partly on these findings).
The issue appears to be the usage of the derivatives of these multipliers as $L^2$
multipliers -- but the derivatives can be unbounded also in this Zygmund situation.

Our example is
\begin{equation}\label{eq:exampleFP}
	m(\xi_1, \xi_2, \xi_3):=\frac{\xi_3}{(\xi_1^2\xi_2^2+\xi_3^2)^{1/2}}.
\end{equation}
We show the following.
\begin{prop}
	Let $m$ be given by \eqref{eq:exampleFP}. Then $m$ satisfies
	\[
		\abs{ \partial^\alpha m(\xi)} \lesssim
		|\xi_1|^{-\alpha_1+\alpha_2} |(|\xi_1|\xi_2, \xi_3)|^{-\alpha_2-\alpha_3}
	\]
	for all multi-indices $\alpha=(\alpha_1, \alpha_2, \alpha_3)$. This means that $m \in \calM_Z^1$, i.e.,
	$m$ is a Fefferman--Pipher multiplier.
	Moreover, $\partial_{\xi_i} m \notin L^\infty$ for all $i=1,2,3$.
\end{prop}

\begin{proof}
	We first notice that
	$$
		\abs{ \partial^\alpha m(\xi)} \lesssim
		|\xi_1|^{-\alpha_1+\alpha_2} |(|\xi_1|\xi_2, \xi_3)|^{-\alpha_2-\alpha_3}
	$$
	is, indeed, the Fefferman--Pipher condition -- while it was not phrased like this in \cite{FEPI},
	it has been spelled out in this form in \cite{HLMV}.

	Of course, for $\alpha=(0,0,0)$ we have $|m(\xi)|\le 1$ and the claimed estimate holds.
	Next, we estimate  $\partial_{\xi_i} m(\xi)$ for $i=1,2,3$.
	Direct computations give the following
	\begin{align*}
		\partial_{\xi_1} m(\xi) & = \xi_3 (-\frac 12) (\xi_1^2\xi_2^2+\xi_3^2)^{-3/2} 2\xi_1 \xi_2^2
		=- \xi_1^{-1}\frac{\xi_1^2\xi_2^2 \xi_3}{ (\xi_1^2\xi_2^2+\xi_3^2)^{3/2}},                                           \\
		\partial_{\xi_2} m(\xi) & = \xi_3 (-\frac 12) (\xi_1^2\xi_2^2+\xi_3^2)^{-3/2} 2\xi_1^2 \xi_2
		=- \xi_1\frac{\xi_1 \xi_2  \xi_3}{ (\xi_1^2\xi_2^2+\xi_3^2)^{3/2}},                                                  \\
		\partial_{\xi_3} m(\xi) & = \xi_3 (-\frac 12) (\xi_1^2\xi_2^2+\xi_3^2)^{-3/2} 2\xi_3+(\xi_1^2\xi_2^2+\xi_3^2)^{-1/2}
		= \frac{\xi_1^2\xi_2^2  }{ (\xi_1^2\xi_2^2+\xi_3^2)^{3/2}}.
	\end{align*}
	It is clear that when
	\[
		(\xi_1, \xi_2, \xi_3)\in [\eps, 2\eps]\times [\eps, 2\eps]\times [\eps^2, 2\eps^2],\qquad \text{where $\eps>0$},
	\]
	we have that $|\partial_{\xi_i} m(\xi)|\sim \eps^{-1}$ for $i=1,2$ and $|\partial_{\xi_3} m(\xi)|\sim \eps^{-2}$.
	This gives that $\partial_{\xi_i} m(\xi)\notin L^\infty$ for all $i=1,2,3$.
	Moreover, since $|(|\xi_1|\xi_2, \xi_3)|=(\xi_1^2\xi_2^2+\xi_3^2)^{1/2} $,
	\[
		|\xi_1\xi_2|\le (\xi_1^2\xi_2^2+\xi_3^2)^{1/2} \qquad \textup{and} \qquad |\xi_3|\le (\xi_1^2\xi_2^2+\xi_3^2)^{1/2},
	\] we have proved the Fefferman--Pipher style derivative estimate whenever $|\alpha|=1$.

	The general case will be proved via induction. We claim that in general
	\begin{equation}\label{eq:induction}
		\partial^\alpha m(\xi)=\xi_1^{-\alpha_1+\alpha_2} |(|\xi_1|\xi_2, \xi_3)|^{-\alpha_2-\alpha_3}\sum_{k\le 2|\alpha|+1}\frac{P_k(\xi_1\xi_2, \xi_3)}{(\xi_1^2\xi_2^2+\xi_3^2)^{k/2}},
	\end{equation}
	where $P_k(t_1, t_2)$ is a polynomial of order $k$. Clearly, for $|\alpha|=1$ \eqref{eq:induction} holds.
	Suppose \eqref{eq:induction} is true for all $\alpha$ with $|\alpha|\le N$.
	We calculate
	\begin{align*}
		 & \partial_{\xi_1}\partial^\alpha m(\xi)                                                                                                                                                                                                                                                  \\
		 & = (-\alpha_1+\alpha_2)\xi_1^{-\alpha_1-1+\alpha_2} |(|\xi_1|\xi_2, \xi_3)|^{-\alpha_2-\alpha_3}\sum_{k\le 2|\alpha|+1}\frac{P_k(\xi_1\xi_2, \xi_3)}{(\xi_1^2\xi_2^2+\xi_3^2)^{k/2}}                                                                                                     \\
		 & \quad+(-\alpha_2-\alpha_3)\xi_1^{-\alpha_1-1+\alpha_2} |(|\xi_1|\xi_2, \xi_3)|^{-\alpha_2-\alpha_3}\sum_{k\le 2|\alpha|+1}\frac{\xi_1^2 \xi_2^2P_k(\xi_1\xi_2, \xi_3)}{(\xi_1^2\xi_2^2+\xi_3^2)^{k/2+1}}                                                                                \\
		 & \quad+ \xi_1^{-\alpha_1-1+\alpha_2} |(|\xi_1|\xi_2, \xi_3)|^{-\alpha_2-\alpha_3}\sum_{k\le 2|\alpha|+1}\Big(\frac{\xi_1\xi_2(\partial_1 P_k)(\xi_1\xi_2, \xi_3)}{(\xi_1^2\xi_2^2+\xi_3^2)^{k/2}}-k \frac{\xi_1^2 \xi_2^2 P_k(\xi_1\xi_2, \xi_3)}{(\xi_1^2\xi_2^2+\xi_3^2)^{k/2+1}}\Big)
	\end{align*}
	which is exactly of the form \eqref{eq:induction}, since $k+2\le 2(|\alpha|+1)+1$.
	Similarly, we have
	\begin{align*}
		 & \partial_{\xi_2}\partial^\alpha m(\xi)                                                                                                                                                                                                                                           \\
		 & = (-\alpha_2-\alpha_3)\xi_1^{-\alpha_1+\alpha_2+1} |(|\xi_1|\xi_2, \xi_3)|^{-\alpha_2-1-\alpha_3}\sum_{k\le 2|\alpha|+1}\frac{\xi_1 \xi_2 P_k(\xi_1\xi_2, \xi_3)}{(\xi_1^2\xi_2^2+\xi_3^2)^{(k+1)/2}}                                                                            \\
		 & \quad+ \xi_1^{-\alpha_1+\alpha_2+1} |(|\xi_1|\xi_2, \xi_3)|^{-\alpha_2-1-\alpha_3}\sum_{k\le 2|\alpha|+1}\Big(\frac{(\partial_1 P_k)(\xi_1\xi_2, \xi_3)}{(\xi_1^2\xi_2^2+\xi_3^2)^{(k-1)/2}}-k\frac{\xi_1 \xi_2 P_k(\xi_1\xi_2, \xi_3)}{(\xi_1^2\xi_2^2+\xi_3^2)^{(k+1)/2}}\Big)
	\end{align*}
	and
	\begin{align*}
		 & \partial_{\xi_3}\partial^\alpha m(\xi)                                                                                                                                                                                                                                    \\
		 & = (-\alpha_2-\alpha_3)\xi_1^{-\alpha_1+\alpha_2} |(|\xi_1|\xi_2, \xi_3)|^{-\alpha_2-\alpha_3-1}\sum_{k\le 2|\alpha|+1}\frac{\xi_3 P_k(\xi_1\xi_2, \xi_3)}{(\xi_1^2\xi_2^2+\xi_3^2)^{(k+1)/2}}                                                                             \\
		 & \quad+ \xi_1^{-\alpha_1+\alpha_2} |(|\xi_1|\xi_2, \xi_3)|^{-\alpha_2-\alpha_3-1}\sum_{k\le 2|\alpha|+1}\Big(\frac{(\partial_2P_k)(\xi_1\xi_2, \xi_3)}{(\xi_1^2\xi_2^2+\xi_3^2)^{(k-1)/2}}-k \frac{\xi_3 P_k(\xi_1\xi_2, \xi_3)}{(\xi_1^2\xi_2^2+\xi_3^2)^{(k+1)/2}}\Big),
	\end{align*}
	which are also of the form \eqref{eq:induction}. Hence,
	we have proved \eqref{eq:induction} for all $\alpha$. The estimate
	\[
		\abs{ \partial^\alpha m(\xi)} \lesssim
		|\xi_1|^{-\alpha_1+\alpha_2} |(|\xi_1|\xi_2, \xi_3)|^{-\alpha_2-\alpha_3}
	\]
	follows then immediately.
\end{proof}

\bibliography{references}

\end{document}